\documentclass[12pt]{article}

\usepackage{layout, graphicx, amsmath, amsthm, amssymb, euscript, enumerate, bbold, bbm, graphicx, fancyhdr,xcolor,mathtools, tikz,url,mathrsfs}
\usepackage{subcaption}
\usetikzlibrary{arrows,shapes}
\usepackage{verbatim}
\usepackage{yfonts} 
\usetikzlibrary{calc,trees,positioning,arrows,chains,shapes.geometric,%
	decorations.pathreplacing,decorations.pathmorphing,shapes,%
	matrix,shapes.symbols,plotmarks,decorations.markings,shadows}

\usepackage[pdftex,bookmarks,colorlinks]{hyperref}
\hypersetup{colorlinks=false}

\makeatletter                                         
\def\th@newremark{\th@remark\thm@headfont{\bfseries}}  
\makeatletter                                          

\theoremstyle{newremark}                               

\newcommand{\Bin}{\mbox{Bin}}
\newcommand{\card}{\mbox{Card}}

\def\N{\mathbb{N}}

\def\ind{ \mathbbm{1} }

\DeclareMathOperator{\length}{length}

\newtheorem{theorem}{Theorem}[section]

\newtheorem{proposition}[theorem]{Proposition}

\newcommand{\VEC}[1]{\mathbf{#1}}

\definecolor{burgundy}{rgb}{0.5, 0.0, 0.13}
\definecolor{linkblue}{RGB}{0,20,128}
\definecolor{linkred}{RGB}{128, 0, 6}
\definecolor{citegreen}{RGB}{46, 126, 42}
\definecolor{urlmagenta}{RGB}{138, 0, 135}
\hypersetup{
  linkcolor  = linkblue,
  citecolor  = citegreen,
  urlcolor   = urlmagenta,
  colorlinks = true,
}
\makeindex
\begin{document}

\title{Crossing bridges between percolation models and Bienaym\'e-Galton-Watson trees}
	\author{  Airam Blancas\thanks{Department of Statistics, ITAM, Mexico.
		Email: {airam.blancas@itam.mx}},  Mar\'ia Clara Fittipaldi\thanks{Facultad de Ciencias, UNAM, Mexico. Email: {mcfittipaldi@ciencias.unam.mx} }, 
 Sara\'i Hern\'andez-Torres\thanks{Instituto de Matem\'aticas, UNAM, Mexico. Email: {saraiht@im.unam.mx} }}
 \date{}
	\maketitle
	\vspace{-.2in}
\maketitle

\begin{abstract}
In this survey, we explore the connections between two areas of probability: percolation theory and population genetic models. Our first goal is to highlight a construction on Bienaym\'e-Galton-Watson trees, which has been described in two different ways: as Bernoulli bond percolation and as neutral mutations. Next, we introduce a novel connection between the Divide-and-Color percolation model and a particular multi-type Bienaym\'e-Galton-Watson tree. We provide a gentle introduction to these topics while presenting an overview of the results that connect them.

\bigskip

\noindent	\textbf{Keywords:} Bienaymé-Galton-Watson trees; Bernoulli percolation; multi-allelic neutral mutations.\\
\textbf{MSC2020 subject classifications:} 
60J80
; 92D25
; 60K35
; 82B43
.
\end{abstract}

\section{Introduction} \label{sec:intro}

Percolation and branching processes are two flourishing areas within probability theory that have received increased attention in recent years. Both fields present rich mathematical problems that have fostered dynamic research communities.
This survey explores connections between these areas. We begin with a brief introduction to the models involved.

A percolation model is associated with dispersion in porous media, where the physical phenomenon of percolation can occur.
This phenomenon refers to the passage of a fluid from the surface to the center of a material.
One of the simplest yet most intriguing percolation models is Bernoulli bond percolation, defined by Broadbent and Hammersley in the 1950s~\cite{BroadbentHammersley}.

In Bernoulli bond percolation, a graph $G$ represents the underlying space. 
A random subgraph of $G$ is generated  by randomly deleting edges of $G$ according to a retention parameter $p \in [0,1]$.  
The remaining edges are considered open, so that $p$ modulates the density of open edges (a precise construction is given in Subsection~\ref{sec:BerPer}).
Referring to the initial description of percolation phenomena, a liquid can advance through an open edge.

Roughly speaking, the percolation event occurs when there is an open connected component of size comparable to the underlying graph $G$. We refer to an open connected component as a \emph{cluster}. When $G$ is an infinite graph, the percolation event corresponds to the existence of an infinite open cluster. 
In particular, we may look at the open cluster connected to the origin $0$. We denote this cluster by $\mathcal{C}_0$, which corresponds to the subset of vertices connected to $0$ by open edges. In the terms described above, if a liquid advances through open edges, $\mathcal{C}_0$ is the set of the vertices reachable from the origin.

Bernoulli percolation exhibits a phase transition between a subcritical phase and a supercritical phase, depending on the density of the open edges.
In the subcritical phase, the density is low enough that the percolation event will not occur with probability one. However, as the density increases, there is a sudden shift, and the percolation event occurs with probability one. In particular, for an infinite graph $G$ and in the supercritical case, the cluster $\mathcal{C}_0$ has a positive probability of being infinite.

Since the introduction of the model, some of the most important questions for the percolation research community have been closely related to conjectures about the system's behavior at the critical parameter, often originating from physics.
For instance, questions about universality and conformal invariance at criticality were initially proposed in the physics literature. We discuss these questions and the achievements in this area in Subsection~\ref{subsec:universalitycritical}.

This survey, however, takes a different direction. Its motivation is to explore and interpret percolation models in ways that relate to branching processes and population dynamics.

The study of branching processes began in 1845 with Bienaymé~\cite{Bienayme} and was advanced in 1873 by Galton and Watson~\cite{GaltonWatson}.
We refer to~\cite{Kendall75} for an account of the origins of this model.
Branching processes model the random behavior of a population's genealogy as successive generations reproduce over time. The main model in this field is the Bienaymé-Galton-Watson process, valued for its simplicity, wide-ranging applications, and flexibility in incorporating variations to account for more realistic scenarios. For example, to represent populations with multiple types of individuals, one can employ the multi-type Bienaymé-Galton-Watson process (detailed in Subsection~\ref{sec:GWfintyp}).

Another approach to incorporating evolutionary dynamics into a branching process is through the Bienaymé-Galton-Watson process with neutral mutations (discussed in Subsection~\ref{sec:GWmut}). Bertoin introduced and analyzed this model in~\cite{Bertoin10, Bertoin09}. In this model, mutations change the genotype of individuals without affecting their reproductive laws, which follow a standard Bienaymé-Galton-Watson process.
Since mutations occur in ancestral lineages, offspring may not inherit the genetic type of their mothers (referred to as an \emph{allele} in biological terms).
Additionally, we assume the population has an infinite number of alleles, meaning each mutation event produces a different allele for a mutant---an individual with a new allele. This dynamic is known as neutral evolution, where reproductive fitness remains unaffected by evolutionary changes, as proposed in Kimura's model of infinite alleles~\cite{Kimura1971}.

The assumption of infinite alleles is prevalent in mutation models, though it may not always be the most suitable assumption. In particular, some works in the biological literature emphasize the importance of working with models that consider a finite number of alleles~\cite{PMID:29030470, MSRM13, SiFit17}.

Finite allele models are particularly suitable for DNA evolution due to the inherent chemistry of DNA, which allows for a finite number of combinations of the nitrogenous bases. Some examples of this were given by Jukes and Cantor~\cite{JukesCantor1969} and Kimura~\cite{Kimura1980, Kimura1981}, among others. A well-known model with these assumptions is the \emph{parent-independent mutation} (PIM) model, an extension of these models in which the type of a mutant does not depend on the type of its mother.
In Subsection~\ref{subsec: BGWfinitemutation}, we introduce a version of the PIM model along a Bienaymé-Galton-Watson tree.

Initially, it may appear that percolation theory and population genetics are unrelated. However, previous work has revealed deep relations through branching processes. These connections go both ways: they unveil large-scale properties of percolation on graphs and properties of mutations along genealogies.

The work of Bertoin is seminal in the applications of Bernoulli percolation to population genetics. In~\cite{Bertoin09}, he considers a subcritical Bienaymé-Galton-Watson process with neutral mutations (infinite alleles model) and decomposes the entire population into clusters of individuals carrying the same allele. The author studied the law of this allelic partition in terms of the distribution of the number of clone-children and the number of mutant-children of a typical individual. In a follow-up work~\cite{Bertoin10}, Bertoin focused on the situation where the initial population is large and the mutation rate small. He proves that, for an appropriate regime, the process of the sizes of the allelic subfamilies converges in distribution to a certain continuous-state branching process in discrete time. 
This last work was extended in two directions by Blancas and Rivero~\cite{BlancasRivero}. In the critical case, they constructed a version of Bertoin's model conditioned not to go extinct. Additionally, they established a version of the limit theorems presented in~\cite{Bertoin10} when the reproduction law has infinite variance and is in the domain of attraction of an $\alpha$-stable distribution, both for the unconditioned process and for the process conditioned to non-extinction.

On the other hand, interpreting Bernoulli percolation on a random tree as neutral mutations along a genealogical tree has proven to be a successful strategy. For instance, Berzunza~\cite{Gabriel2020} studied Bernoulli percolation on trees associated with the genealogy of Crump-Mode-Jagers processes, which model populations where individuals may give birth at different points in time during their lives. The main result in~\cite{Gabriel2020} is the existence of a giant percolation cluster for large Crump-Mode-Jagers trees. In~\cite{BertoinUB}, Bertoin and Uribe Bravo considered supercritical Bernoulli bond percolation on large scale-free trees, closely related to Yule processes. Using the neutral mutations interpretation, the authors in~\cite{BertoinUB} obtained a weak limit theorem for the sizes of the clusters.

In this survey, we focus on the connections between percolation theory and the two extensions of Bienaymé-Galton-Watson (BGW) processes discussed above: the multi-type Bienaymé-Galton-Watson process and the Bienaymé-Galton-Watson process with neutral mutations.

Our starting point is the key observation in~\cite{Bertoin09}: a Bienaymé-Galton-Watson process with neutral mutations and an \emph{infinite number of alleles} is equivalent to Bernoulli percolation on a Bienaymé-Galton-Watson tree. This representation serves as the primary bridge between the two topics covered in this survey, and we aim to explain it carefully.

As noted earlier, in some cases, a finite number of alleles provides a more realistic model. Establishing the corresponding connection between a Bienaymé-Galton-Watson process with neutral mutations and a \emph{finite number of alleles} requires a variation of Bernoulli percolation known as Divide-and-Color percolation, introduced by Häggström in~\cite{Haggstrom}.

Divide-and-Color percolation is a simple and natural model for coloring the vertices of a graph, with its construction detailed in Subsection~\ref{sec:DaC}. In this survey, we propose a novel interpretation of Divide-and-Color percolation as a model for neutral mutations along a genealogy when the number of alleles is finite.

The remainder of this survey is organized as follows.
In Section~\ref{sec:trees}, we provide essential definitions of trees, treating them both as graphs and as genealogical structures.
Section~\ref{sec:branching} introduces the Bienaymé-Galton-Watson (BGW) process and its extensions, including the multi-type BGW process, the BGW process with neutral mutations and infinite alleles, and the BGW process with neutral mutations and finite alleles. For additional background  on branching processes, we recommend the classical reference by Athreya and Ney~\cite{AthreyaNey} and the more recent notes by Lambert~\cite{AmauryNotas2008}.
Section~\ref{sec:percolation} focuses on Bernoulli percolation on finite and infinite graphs, along with the Divide-and-Color percolation model. The standard reference for Bernoulli percolation is Grimmett~\cite{GrimmettPercolation}, while percolation theory on trees is thoroughly discussed in~\cite{lyons2017probability}.
Finally, in Section~\ref{sec:connections}, we discuss the connections between the branching processes and percolation models introduced in the preceding sections.

\section{Trees}  \label{sec:trees}

In this section, we look at trees as mathematical objects from two perspectives. From a combinatorial viewpoint, we first define a tree as a graph, following the standard notation presented in Lyons and Peres~\cite{lyons2017probability}.  Afterwards, we consider a tree as a genealogical structure. For further details on the latter perspective, we refer to Abraham and Delmas~\cite{AbrahamDelmas15} and Duquesne and Le Gall~\cite{DuquesneLeGall}.

\subsection{Trees as graphs} \label{sec:graphs}

A \emph{graph} is a pair $G = (V, E)$, where $V$ is the set of \emph{vertices} and $E$ is a subset of $V \times V$ called the \emph{edge set}. We consider the edges of $E$ to be \emph{unoriented}, and accordingly, we use the notation $\{x, y\}$ for an element of $E$.
If $\{x,y\}\in E$, then we call $x$ and $y$ \emph{adjacent} or \emph{neighbors}, and we write $x \sim y$. The \emph{degree} of a vertex is the number of its neighbors. Whenever we have more than one graph under consideration, we shall distinguish the vertex and edge sets of $G$ by writing $V(G)$ and $E(G)$. A graph $G$ is \emph{finite} if the cardinality $\vert V \vert < \infty $; otherwise, we say that $G$ is an infinite graph. An infinite graph is \emph{locally finite} if every vertex has finite degree. From now on, we assume that all infinite graphs are locally finite as well.

A \emph{subgraph} $G'$ of $G$ is a graph whose vertex set is a subset of $V(G)$ and whose edge set is a subset of $E(G)$. If $G'$ satisfies that its vertex subset is $V$, we refer to 
$G'$ as a \emph{spanning subgraph}.

A graph \emph{isomorphism} between $G = (V,E)$ and $G'= (V',E')$  is a bijection  $\varphi : V \rightarrow V'$ for which $\varphi(x) \sim \varphi(v)$ if and only if $ x \sim v$. 
An isomorphism between a graph $G$ and itself is called an \emph{automorphism}. 
A graph is \emph{transitive} if for any $u, v \in V$ there exists an automorphism $\varphi_{u,v}$ such that $\varphi_{u,v} (u) = v$.

A \emph{path} in a graph is a sequence of vertices $ [v_0, v_1, \ldots ] $  
such that each successive pair of vertices forms an edge in the graph.
A finite path is said to \emph{connect} its first and last vertices.  
When a path does not pass through any vertex (resp., edge) more than once, we call it \emph{simple}. 
The \emph{length} of a finite path $ [v_0, v_1, \ldots, v_n] $ is the number of edges connecting each pair of vertices in the sequence, so we write $ \length([v_0, v_1, \ldots, v_n]) = n $.
We say that such a path is between $v_0$ and $v_n$.

A finite path with at least one edge, and whose first and last vertices are the same, is called a \emph{cycle}. 
A cycle is \emph{simple} if no pair of vertices is the same except for its first and last ones. A graph is \emph{connected} if, for each pair $x\neq y$ of its vertices, there is a path connecting $x$ to $y$. 

Let $G = (V,E)$ be a finite, or locally finite, connected graph. The \emph{graph metric}  $d_G : V \times V \rightarrow [0, \infty)$ is	defined by 
\begin{equation} \label{eq:graphMetric}
	d_G (x,y) \coloneqq \inf \{  \length({\mathbf{p}}) \: : \: \mathbf{p} \text{ is a simple path between } x \text{ and } y \}. 
\end{equation}
The closed ball of radius $r > 0$ with center at $x$ is the set of vertices
\[
	B_G (x, r) \coloneqq \{  y \in V \: : \: \ d_G (x,y) \leq r  \}.
\]
If $G$ is a transitive graph, any two closed balls of radius $r > 0$ are isomorphic as graphs.

A graph with no cycles is called a \emph{forest}, and a connected forest is a \emph{tree}.
We will work with \emph{rooted trees}, meaning that some vertex is designated as the root, denoted by $\emptyset$. 
Once we set a root, we can label the tree using a partial order for its vertices. We are interested in a labeling that conveys genealogical data, which we introduce in the next subsection.

\subsection{Trees as genealogical structures}

In this note, we will work in the following setting. We imagine that the tree is growing away from its root, and its growth specifies a population genealogy: the root has edges leading to its children, then these vertices branch out to their own children, and we advance in this way throughout the generations.
The ancestry of an individual in the tree determines its label in the Ulam-Harris notation, which we introduce below.

Let us define the set of finite sequences of positive integers as
\[ 
	\mathcal{U}=\bigcup\limits_{n\geq 0}\mathbb{N}^n,
\]
 following the convention $\mathbbm{N}^0 = \{\emptyset\}$. 
 For $n\geq 1$ and $v=(v_1 \dots v_n) \in \mathcal{U}$, we denote by $\Vert v\Vert=n$ the \emph{length} of $v$, where the root satisfies $\Vert \emptyset \Vert=0$. Let $u,v \in \mathcal{U}$. We say that $v $  is an \emph{ancestor of} the vertex $u$, and write $v \prec u $, if there exists $w \in \mathcal{U}$, $w\neq \emptyset$, such that we can write $u=vw$, the concatenation of the two sequences. Observe that  $u\preceq u$ for any $u\in \mathcal{U}$. This ancestry relation defines a lexicographic order on $\mathcal{U}$, which we will use to set the index of each vertex on a given tree, starting from the root. 
 Within this framework, a \emph{genealogical tree}  $\mathbf{t}$ is a subset of $\mathcal{U}$ satisfying the following conditions:
\begin{enumerate}
	\item  $\emptyset \in \mathbf{t}$. 
	\item  If $u \in \mathbf{t}$, then $A_u \coloneqq \{v \in \mathcal{U}: v\prec u\} \subset \mathbf{t}$.
	\item  For every $u \in \mathbf{t}$, there exists $k_u(\mathbf{t}) \in \mathbb{N}_0 \cup \{+\infty\}$ such that for every $i \in \mathbb{N}$, $ui \in \mathbf{t}$ if and only if $1 \leq i\leq k_u(\mathbf{t})$.
\end{enumerate} 
The set $A_u$ corresponds to the \emph{set of ancestors} of $u$, and the integer $k_u(\textbf{t})$ indicates the number of offspring of the vertex $u \in \textbf{t}$. See Figure \ref{fig1: UH BGW } for an example.

For any $u \in \textbf{t}$, we define the subtree $\textbf{t}_u$ of $\textbf{t}$ with root $u$ as:
\[
 \mathbf{t}_{u} \coloneqq \{v \in \mathcal{U}: uv \in \textbf{t}\}.
\]
Then, the tree $\mathbf{t}_{u}$ corresponds to the genealogical tree of the individual $u$ and its descendants.

We denote by $\mathbbm{T}$ the set of trees without vertices of infinite degree (also known as locally finite trees):
\[
\mathbbm{T} = \{\textbf{t} : k_u (\textbf{t}) < +\infty \quad \forall u \in  \textbf{t}\}.
\] 
For $\textbf{t} \in \mathbbm{T}$, we denote by $\vert \textbf{t} \vert = \card \,\textbf{t}$ the cardinality of its vertex set. Let $\mathbbm{T}_0$ be the set of finite trees:
\[\mathbbm{T}_0 = \{\textbf{t} \in \mathbbm{T} : \vert \textbf{t}\vert < +\infty \}.
\]

\section{Genealogical structures for branching processes} \label{sec:branching}

\subsection{Bienaym\'e-Galton-Watson process}\label{sec:GW}

Consider a population evolving in discrete time, where each time step denotes a generation. At each given time, every individual dies and gives birth to a random number $\xi$ of children in the next generation, according to an \textit{offspring distribution} $\mu = (\mu(k), k \in \mathbb{N}_0)$, so that $\mathbb{P}(\xi = k) = \mu(k)$.

For each $j \in \mathbb{N}$ and $n\in \mathbb{N}_0$, we write $\xi_{n+1, j}$  for the number of children alive in the $(n+1)$-th generation, of the $j$-th individual in generation $n$.
The random variables $((\xi_{n+1, j}; j \in \mathbb{N}), n \in \mathbb{N})$ are independent and identically distributed  with common distribution $\mu$. In order to avoid trivial cases, we assume throughout that $\mu(0) + \mu(1)<1$ and $\mu(k)\neq 1$ for any $k$. A Bienaym\'e-Galton-Watson (BGW)  process with offspring distribution $\mu$ is a discrete-time Markov chain $(X_n, n \in \mathbb{N}_0)$  defined as $X_0 = 1$ and
\[
X_{n+1}= \sum\limits_{j=1}^{X_{n}}\xi_{n+1,j}, \qquad \forall n \in \mathbb{N}_0,
\]
so that  $X_n$  is the size of the population at generation $n$. 
The law of a BGW process starting from $x$ individuals (i.e $X_0 = x$) is denoted by $\mathbb{P}_x$.

Let $(X_n(x), n \in \mathbb{N}_0)$ be a BGW process starting from $x$ individuals. One of the key properties of a BGW process is the \emph{branching property}. It establishes that 
\[
X(x+y)\overset{d}{=}\bar{X}(x) +\widetilde{X}(y),
\]
where $\bar{X}$ and $\widetilde{X}$ are independent copies of $X$ starting from $x$ and $y$ individuals, respectively. In other words, for each $x,y \in \mathbb{N}_0$, the probability measure $\mathbb{P}_{x+y}$ is equal to the convolution of  $\mathbb{P}_{x}$ and  $\mathbb{P}_{y}$ .

A \emph{BGW tree}, denoted as $\mathcal{T}$,  is the genealogical tree resulting from the BGW process originating from a single individual. Therefore, $\mathcal{T}$ is a $\mathbb{T}$-valued random variable, with a root denoted by $\emptyset$ to signify the initial individual, and edges establishing connections between parents and their children. For each vertex $v$ within $\mathcal{T}$, $k_{v} (\mathcal{T})$ denotes the number of offspring it has, following a distribution $\mu$.

We follow the Ulam-Harris notation to label the genealogical tree $\mathcal{T}$. According to this notation, an individual indexed by $u=(u_1\dots u_{n} u_{n+1})$ in the $(n+1)$-th generation is identified as the $u_{n+1}$-th child of the individual $(u_1\dots u_{n})$ in the $n$-th generation. This labeling allows us to trace backward in time along the genealogy until reaching the $u_1$-th individual in the first generation, who is a child of the root denoted by $\emptyset$.

As a consequence of the definition of the BGW process, the subtree $\mathcal{T}_{u}$, rooted at $u\in \mathcal{T}$, is a BGW tree with the same distribution as $\mathcal{T}$. It is worth noting that $\mathcal{T}_{u}$ represents the genealogical tree associated with a BGW process $(\widetilde{X}_n, n\geq 0)$ starting from the individual $u$.

By the branching property,  conditioned on the event $\{k_{\emptyset}(\mathcal{T})=n\}$, the subtrees $(\mathcal{T}^{(1)},\dots,\mathcal{T}^{(n)} )$ are independent and follow the distribution of the original tree $\mathcal{T}$.

\begin{figure}[h]
	\begin{center}
		\begin{tikzpicture}
			[
			level 1/.style={sibling distance=40mm},
			level 2/.style={sibling distance=30mm},
			level 3/.style={sibling distance=12mm},
			level 4/.style={sibling distance=10mm},
			every node/.style={circle, draw=black,thin, minimum size = 1 cm}]
			\node{ $\emptyset$}
			child {node {1}
				child {node {11}
					child {node {111}}
					child {node {112}}
					child {node {113}}}
				child {node {12}
					child {node {121}}}}
			child {node {2}
				child {node {21}
					child {node {211}}
					child {node {212}}}};
		\end{tikzpicture}
		\caption{
		This figure shows an example of a BGW tree with Ulam-Harris labeling. In this instance, the offspring distribution is defined as  $\mu(k)=\tfrac{1}{3}\ind_{ \{1,2,3\} }(k)$.}
		\label{fig1: UH BGW }
\end{center}
\end{figure}
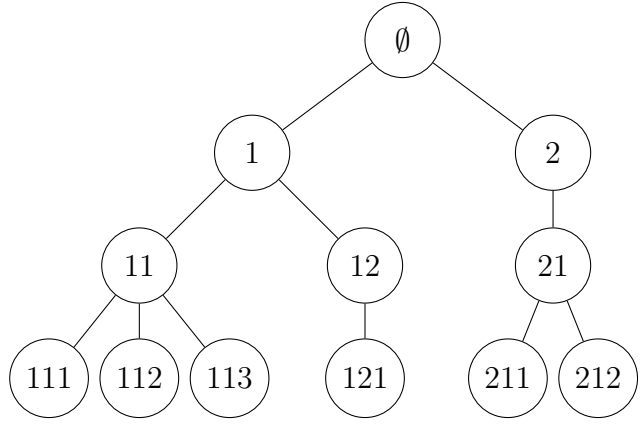

\subsection{Multi-type BGW processes }\label{sec:GWfintyp}

 We now aim to describe a branching population with $d$ types of individuals; these types are indexed by the set $[d]\coloneqq \{1, 2\dots, d\}$. For each $i \in [d]$, individuals of type $i$ will produce offspring of possibly every type, independently of each other. To avoid trivial cases, we assume that the event where all children of an individual inherit its type occurs with probability less than one.
Formally, let $\boldsymbol{\xi}^{(i)}=(\xi^{(i,j)})_{j \in [d]}$ denote the random vector where the  $j$-th entry is the number of children that an individual of type $i$ gives birth. Its distribution is given by $\mu_i$, i.e.,
\begin{equation}\label{eq:muloffdist}
	\mu_i( {\VEC{v}} ) = \mathbb{P}( \xi^{(i,1)}= v_1,  \xi^{(i,2)} = v_2 , \dots, \xi^{(i,d)} = v_d ), \qquad \VEC{v} = (v_1, \ldots, v_d) \in \N_0^d.
\end{equation}

We define $\boldsymbol{\xi}^{(i)}_{n+1,k}=(\xi^{(i,1)}_{n+1,k},\xi^{(i,2)}_{n+1,k}, \dots, \xi^{(i,d)}_{n+1,k} )$ as the number of children in the $(n+1)$-th generation born to the $k$-th individual of type $i$ in generation $n$.
Let us assume that, for each $i \in [d]$, $( (\boldsymbol{\xi}^{(i)}_{n+1,k} , k\in \mathbb{N}), n \in \mathbb{N}_0 )$ are independent random vectors with common distribution $\mu_i$. In other words, $\mu_i( {\VEC{v}} )$  is the probability that such individual gives birth $\vert\VEC{v}\vert=\sum_{j=1}^d v_j$ children, of which $v_j$ are of type $j$. 
 
Let $Y_n^{(i)}$ denote the number of individuals of type $i$ in the $n$-th generation. Without loss of generality, we assume that the population starts with a single individual of type $i$.
Thus, $(\VEC{Y}_n = (Y_n^{(1)}, \dots, Y_n^{(d)}), n \in \mathbb{N}_0)$, with law $\mathbb{P}$, is a discrete-time Markov chain, recursively defined as ${\VEC{Y}}_0 = \VEC{e}_i$, where $\VEC{e}_i$ denotes the canonical vector of $\mathbb{N}_0^d$, and
\begin{equation}\label{eq:multipopsize}
Y_{n+1}^{(j)}  =  \sum_{i=1}^{d} \sum_{k=1}^{Y_n^{(i)}} \xi^{(i,j)}_{n+1,k} , \qquad  i\in [d], \quad n\in \N_0. 
\end{equation}

The process $(\VEC{Y}_n, n \in \mathbb{N}_0)$ is known as a \emph{multi-type BGW} process with offspring distribution $\boldsymbol{\mu}=( \mu_1, \dots, \mu_d)$. It is clear that when $d=1$ the process coincides with the BGW process.

As in the one-dimensional case, the multi-type BGW process $\VEC{Y}$ satisfies the \emph{branching property}, i.e.,
\[
	\VEC{Y}(\VEC{u}+\VEC{v})\overset{d}{=}\bar{\VEC{Y}}(\VEC{u}) +\widetilde{\VEC{Y}}(\VEC{v}),
\]
where $\bar{\VEC{Y}}$ and $\widetilde{\VEC{Y}}$ are independent copies of $\VEC{Y}$, starting from $\VEC{u} = (u_1, \dots, u_d)$ and $\VEC{v} = (v_1, \dots, v_d)$ individuals, respectively.
As before, this implies that the probability measure $\mathbb{P}_{\VEC{u}+\VEC{v}}$ is equal to the convolution of  $\mathbb{P}_{\VEC{u}}$ and  $\mathbb{P}_{\VEC{v}}$ for any $\VEC{u},\VEC{v} \in \mathbb{N}_0^d$.

A \emph{multi-type BGW tree} $\boldsymbol{\mathcal{T}}^i$ is the genealogical tree associated with the multi-type BGW process $(\VEC{Y}_n , n\in \N)$, starting from an individual of type $i$.
To include the individual type information, we label $\boldsymbol{\mathcal{T}}^i$ according to a  generalization of the Ulam-Harris notation. We identify each vertex with a pair $(i,u)$,  where the first coordinate  $i\in [d]$ indicates the type and $u\in \mathbb{U}$, as before.  See Figure~\ref{fig:2} for an example, where the type is indicated by the vertex color.
Note that the multi-type BGW tree satisfies the branching property in a manner similar to the BGW tree with one type.

\begin{figure}[h]
	\centering
		\begin{tikzpicture} 
			[
			level 1/.style={sibling distance=65mm},
			level 2/.style={sibling distance=21mm},
			level 3/.style={sibling distance=11mm},
			every node/.style={circle, draw=cyan!50,line width=1.5pt, minimum size = 1 cm},
			bluen/.style={circle, draw=blue!50,line width=2pt, minimum size = 1 cm},
			greenn/.style={circle, draw=green!50,line width=2pt, minimum size = 1 cm},
			emph/.style={edge from parent/.style={draw,line width=1.5pt,-,black!50}},
			rede/.style={edge from parent/.style={draw,line width=1.5pt,-,red!50}},
			bluee/.style={edge from parent/.style={draw,line width=5pt,-,blue!50}},
			greene/.style={edge from parent/.style={draw,line width=5pt,-,green!50}},
			]
			\node {$\emptyset$}
			child{node {$1$}
				child {node {$11$}
					child {node {$111$}}
					child {node[bluen] {112}}
					child {node[greenn] {113}}}
				child {node {12}}
				child {node[bluen] {13}
					child {node {131}}
					child {node[bluen] {132}}
					child {node[bluen] {133}}}}
			child {node[bluen] {2}
				child {node[bluen] {21}
					child {node[greenn] {212}}}
				child {node {22}}
				child {node[greenn] {23}
					child {node[bluen] {231}}
					child {node[greenn] {232}}
					child {node {233}}}};
		\end{tikzpicture}
		\caption{ 
		An example of a genealogical tree for our multi-type process is shown. In this case, we label the tree using the usual Ulam-Harris notation, but we include a node color that represents the \emph{type} of each individual, using blue for type 1, purple for type 2, and green for type 3.} \label{fig:2}
\end{figure}
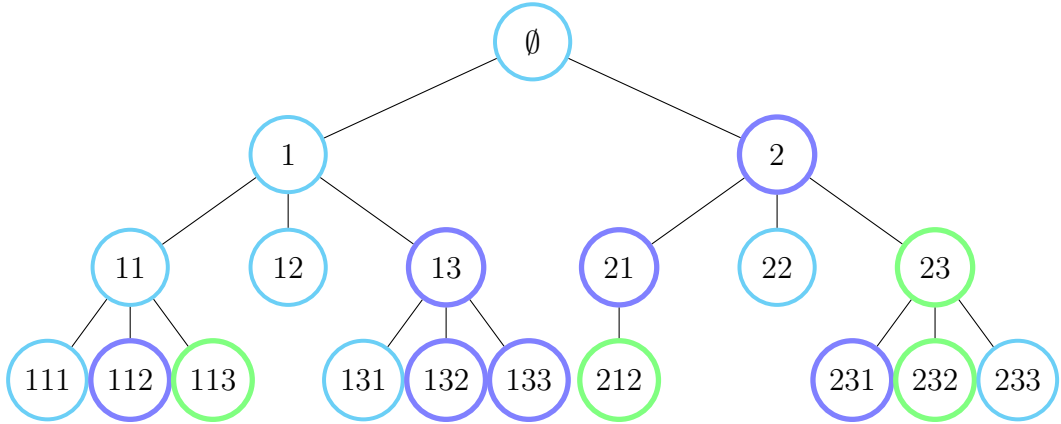

\subsection{BGW processes with neutral mutations and infinite alleles}\label{sec:GWmut}

As a further extension of the BGW model, Bertoin~\cite{Bertoin09, Bertoin10} studied the Galton-Watson process with neutral mutations, which we will call the \emph{BGW process with infinite neutral mutations}. This model assumes that mutations modify the genotype of individuals but do not affect the population dynamics, which evolve according to the BGW process described in Section~\ref{sec:GW}. Since mutations appear in the ancestral lines of the population, each individual may have children that do not necessarily inherit its genetic type. From a biological perspective, a genetic type corresponds to an allele. Moreover, we assume that the population has an infinite number of alleles, meaning that each mutation event gives rise to a different type. 

Following the terminology of \cite{Bertoin10}, the children of individuals that share the same type as their mother are called \emph{clones}, while those of a different type are called \emph{mutants}.
Additionally, $\xi^{(c)}$ and $\xi^{(m)}$ are non-negative integer-valued random variables, not necessarily independent, which describe the number of clones and mutants, respectively, of a given individual. These are distributed according to $\mu^{(c)}$ and $\mu^{(m)}$, respectively. Hence, any individual (whether a mutant or not) has a random number of children, $\xi^{(+)} \coloneqq \xi^{(c)} + \xi^{(m)}$, independently of all other individuals, with distribution  $\mu$. 
In particular, we are interested in the case where mutations affect each child with a fixed probability
$r$, independently of its siblings. 
More precisely, the conditional distribution of $\xi^{(m)}$ given $\{\xi^{(+) }=v_1+v_2 \}$ is binomial with parameters $(v_1+v_2, r)$, i.e.
\begin{equation}\label{eq:infmutlaw}
	\mathbb{P}( \xi^{(c)}=v_1, \xi^{(m)}=v_2)={ \binom{v_1+v_2 }{v_1}}(1-r)^{v_1}r^{v_2}.
\end{equation}

It is possible to view the BGW process with infinite neutral mutations as a multi-type BGW process with countably many types. To be precise, consider a population with types indexed by  $\mathbb{N}_0$, so that whenever individuals of type $i$ reproduce, they give birth to individuals of type $i$ (clones) or type $i+j$ , for some $j \in \mathbb{N}$, where type  $i+j$  has never been seen in the population (mutants). Let $[d_n]$ be the set of types observed up to generation $n$. Hence, if $Y_n^{(i)}$ denotes the number of type $i$ individuals in the $n$-th generation,  we have the following:
\[
	\begin{cases}
		Y_{n+1}^{(i)} =   \sum_{h=1}^{Y_n^{(i)}} \xi^{(c)}_{n+1,h} &\quad \text{for any } i\in[d_n], \\
		\\
		Y_{n+1}^{(d_n+i(h))} =  \xi^{(m)}_{n+1,h} &\quad h=1,\dots, \sum_{i=0}^{d_n} Y_{n}^{(i)}  \text{ if  }  \xi^{(m)}_{n+1,h} \neq 0,
	\end{cases}
\]
where $\{(\xi^{(c)}_{n+1,h}, h \in \mathbb{N}), n \in \mathbb{N}\}$ are independent copies of $\xi^{(c)}$, and $i(h)= \sum_{k=1}^h \ind_{ \{\xi^{(m)}_{n+1,k} \neq 0 \} }$. 
See Figure~\ref{fig:infiniteAllele} for an example of this process.

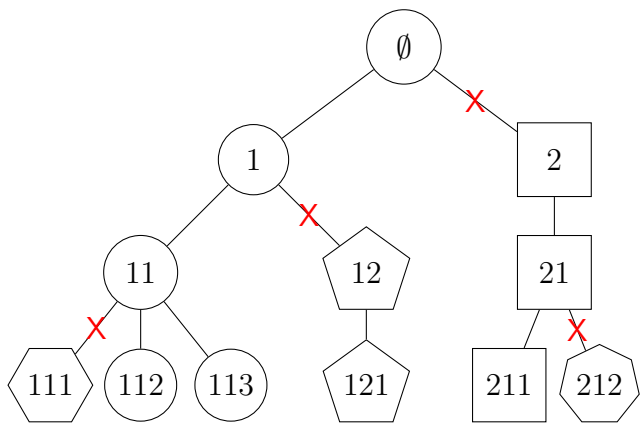
\begin{figure}[ht] 
\centering
	\begin{tikzpicture}
	[
	level 1/.style={sibling distance=40mm},
	level 2/.style={sibling distance=30mm},
	level 3/.style={sibling distance=12mm},
	level 4/.style={sibling distance=10mm},]
	\node[circle,draw,inner sep=6]{ $\emptyset$}
	child {node[circle,draw,inner sep=6] {1}
		child {node[circle,draw,inner sep=5] {11}
			child {node[regular polygon,
				regular polygon sides=6,draw,inner sep=1] {111}edge from parent node {\textcolor{red}{{\sf X}}}}
			child {node[circle,draw,inner sep=3] {112}}
			child {node[circle,draw,inner sep=3] {113}}}
		child {node[regular polygon,
			regular polygon sides=5,draw,inner sep=4] {12}
			child {node[regular polygon,
				regular polygon sides=5,draw,inner sep=1] {121}}edge from parent node {\textcolor{red}{{\sf X}}}}}
	child {node[regular polygon,
		regular polygon sides=4,draw,inner sep=6] {2}
		child {node[regular polygon,
			regular polygon sides=4,draw,inner sep=4] {21}
			child {node[regular polygon,
				regular polygon sides=4,draw,inner sep=1] {211}}
			child {node[regular polygon,
				regular polygon sides=7,draw,inner sep=1] {212}edge from parent node {\textcolor{red}{{\sf X}}}}}edge from parent node {\textcolor{red}{{\sf X}}}};
\end{tikzpicture}
	\caption{ 
	An example of a BGW tree with infinite allele-type mutations. This tree corresponds to the one in Figure~\ref{fig1: UH BGW } with  with $r=0.5$.  Note that after each new mutation, a new type of individual appears (a type that has never been seen before).}
	\label{fig3: inf alleles} \label{fig:infiniteAllele}
\end{figure}

\subsection{BGW processes with neutral mutations and finite alleles} \label{subsec: BGWfinitemutation}

We are also interested in studying BGW processes with neutral mutations in the case of finitely many alleles. In this framework, we propose two models of multi-allelic neutral mutation for a BGW process. First, we introduce the mother-dependent mutation model, where a mutant cannot have the same type as its mother. In the second model, we drop this assumption and introduce the mother-independent mutation model.

\subsubsection{Mother-dependent mutation model} \label{subsub: MDM}

In the same framework as~\cite{Bertoin10},  we consider neutral mutations in a BGW process with probability $r$, assuming that the population has finitely many alleles.
Formally, we can view the \emph{mother-dependent mutation model} (MDM) as a multi-type BGW process, as introduced in Section~\ref{sec:GWfintyp}. In particular, we have the  discrete-time Markov chain $(\VEC{Y}_t, t\geq 0)$ as in \eqref{eq:multipopsize}, where for $i \in [d]$, the offspring distribution $\mu_i$ defined in~\eqref{eq:muloffdist} is multinomial. That is, for $ \VEC{v} = (v_1, \ldots, v_d) \in \mathbb{N}_0^d $,
\begin{equation}\label{eq: mothdepproba}
	\begin{split}
		\mu_i(\VEC{v})&= \mu(|\VEC{v}|) \binom{|\VEC{v}|}{v_1,\dots,v_d} (1-r)^{v_i} \prod\limits_{j \neq i}\left(\dfrac{r}{d-1}\right)^{v_j}\\
		&= 
		\mu(|\VEC{v}|) 
		\binom{|\VEC{v}|}{v_i} (1-r)^{v_i} r^{|\VEC{v}|-v_i} 
		\binom{|\VEC{v}| - v_i}{v_1,\dots,v_{i-1},v_{i+1},\dots,v_d}\left(\dfrac{1}{d-1}\right)^{|\VEC{v}|-v_i},
	\end{split}
\end{equation}
where $r \in [0,1]$ and $|\VEC{v}|=\sum\limits_{j=1}^d v_j$. Here, $\mu$ is the probability measure on $\N_0$ associated with 
\begin{equation} \label{eq:offspringdist1}
	\xi^{(+)}\overset{d}{=}\xi^{(+,i)} =\sum_{j=1}^{d} \xi^{(i,j)},
\end{equation} 
the total number of children of any individual, independently of its type.

Hence, the evolution described above corresponds to the following mechanism: at each generation, every individual reproduces with probability 
 $\left(\mu(k), k \in \N_0\right)$, independently of others. 
Each child will be a copy of its mother with probability $1-r$, independently of their siblings.
If it is not, its new type will be chosen uniformly from among the  $d-1$ remaining types. 
For an example of a realization of this process, see Figure~\ref{fig:fig4}.

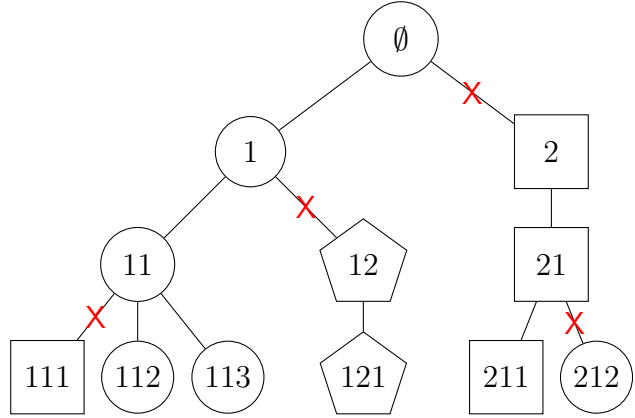
\begin{figure}[ht] \label{fig4: finite alleles}
	\centering
	\begin{tikzpicture}
	[
	level 1/.style={sibling distance=40mm},
	level 2/.style={sibling distance=30mm},
	level 3/.style={sibling distance=12mm},
	level 4/.style={sibling distance=10mm},]
	\node[circle,draw,inner sep=6]{ $\emptyset$}
	child {node[circle,draw,inner sep=6] {1}
		child {node[circle,draw,inner sep=5] {11}
			child {node[regular polygon,
				regular polygon sides=4,draw,inner sep=1] {111}edge from parent node {\textcolor{red}{{\sf X}}}}
			child {node[circle,draw,inner sep=3] {112}}
			child {node[circle,draw,inner sep=3] {113}}}
		child {node[regular polygon,
			regular polygon sides=5,draw,inner sep=4] {12}
			child {node[regular polygon,
				regular polygon sides=5,draw,inner sep=1] {121}}edge from parent node {\textcolor{red}{{\sf X}}}}}
	child {node[regular polygon,
		regular polygon sides=4,draw,inner sep=6] {2}
		child {node[regular polygon,
			regular polygon sides=4,draw,inner sep=4] {21}
			child {node[regular polygon,
				regular polygon sides=4,draw,inner sep=1] {211}}
			child {node[circle,draw,inner sep=3] {212}edge from parent node {\textcolor{red}{{\sf X}}}}}edge from parent node {\textcolor{red}{{\sf X}}}};
\end{tikzpicture}
	\caption{An example of a BGW tree with mother-dependent finite allele mutations. This tree corresponds to the tree in Figure~\ref{fig1: UH BGW } with $r=0.5$ and three types represented as circle, square, and pentagon. Note that in this case, each mutation event produces an individual whose type is chosen uniformly among all available types, excluding its mother's type.} \label{fig:fig4}
\end{figure}


\subsubsection{Mother-independent mutation model} \label{subsub: MIM} \label{sec:mother-independent}

In the \emph{mother-independent mutation model}  (MIM), mutant children are allowed to have the same type as their mother. Similar to the mother-dependent mutation model (MDM), the population evolves as a discrete-time Markov chain $(\VEC{Z}_t, t\geq 0)$, as described in  \eqref{eq:multipopsize}. For $i \in [d]$, the offspring distribution $\mu_i$ defined in \eqref{eq:muloffdist} is multinomial and given by:
\begin{equation}\label{eq:mothindproba}
		\mu_i(\VEC{v})=
            \mu(|\VEC{v}|)
            \sum_{k = 1}^{v_i}
           \binom{|\VEC{v}|}{k} (1-r)^{k}
           r^{|\VEC{v}|-k}  
           \binom{|\VEC{v}|-k}{v_1, \dots , v_i - k , \dots, v_d}
           \left(\dfrac{1}{d}\right)^{ |\VEC{v}| - k },
\end{equation}
for each $ \vert \VEC{v} \vert = (v_1, \ldots, v_d) \in \mathbb{N}_0^d $.
Here, the first term corresponds to the offspring distribution in~\eqref{eq:offspringdist1}, the index $k$ represents the number of clones of type $i$, and the last two factors represent the coloring mechanism, where $v_i - k$  mutant children randomly choose their types, including the possibility of matching their mother’s type. Figure~\ref{fig: mthind mutation} illustrates an example of this model.

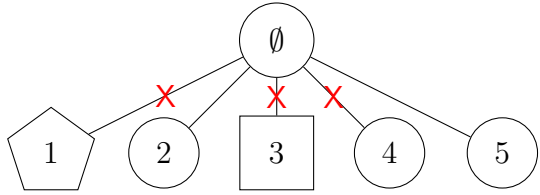
\begin{figure}[ht] \label{fig5: finite alleles}
	\centering
	\begin{tikzpicture}
	[
	level 1/.style={sibling distance=15mm},
	level 2/.style={sibling distance=30mm}
        ]
	\node[circle,draw,inner sep=6]{$\emptyset$}
	child {node[regular polygon,
		regular polygon sides=5,draw,inner sep=6] {1}edge from parent node {\textcolor{red}{{\sf X}}}}
	child {node[circle,draw,inner sep=6] {2}}
	child {node[regular polygon,
		regular polygon sides=4,draw,inner sep=6] {3}edge from parent node {\textcolor{red}{{\sf X}}}}
	child {node[circle,draw,inner sep=6] {4}edge from parent node {\textcolor{red}{{\sf X}}}}
	child {node[circle,draw,inner sep=6] {5}};
\end{tikzpicture}
	\caption{ An illustration of a branching event for the mother-independent mutation (MIM) model in the case where we have three types (circle, square and pentagon). For this model, each mutant child chooses its type uniformly from the set of all possible types.}

\label{fig: mthind mutation}
\end{figure}

\subsection{Phase Transition for branching processes}

\subsubsection{Extinction for BGW processes} \label{subsection:extinctionBGW}

Given a BGW  process $(X_n, n\in \mathbb{N}_0)$, \emph{extinction} occurs  if $X_n=0$ for some generation $n$. Note that if $X_n=0$, then $X_{n+k}=0$ for every $k \in \mathbb{N}$. We denote this event by $\{ \text{Ext} \}$, and thus,
\[
	\{ \text{Ext} \} = \{\text{there exists } n \in \mathbb{N} : X_n = 0\} = \bigcup\limits_{n \in \mathbb{N}}\{X_n = 0\}= \lim\limits_{n \rightarrow +\infty} \{X_n = 0\},
\]
as $\left(\{X_n = 0\}\right)_{n\in \mathbb{N}}$  are increasing events. We point out that the extinction event in the BGW process corresponds to the event that $\mathcal{T}$ is a finite tree, i.e.,  $ \{ \text{Ext} \} = \{\mathcal{T} \in \mathbb{T}_0\}$.

The generating function of $\xi$ is given by 
\[
	f(s)= \sum\limits_{k\in \mathbb{N}_0}s^k\mu(k), \qquad s \in [0,1],
\] 
and its mean is $\mathbb{E}(\xi)=m$. It is well-known that the fixed-point equation $f(s) = s$ has at most two solutions in $[0,1]$, and the extinction probability $q\coloneqq\mathbb{P}(\text{Ext})$ is the smallest root of this equation. In fact, BGW processes display a phase transition, from almost surely extinction ($q=1$) if $m<1$, to positive probability of survival ($q<1$) if $m>1$. We say that a BGW process $(X_n, n \in \mathbb{N}_0)$ (resp. a BGW tree $\mathcal{T}$) is  \emph{subcritical}, \emph{critical} or \emph{supercritical} depending on  $m<1$,  $m=1$ or  $m>1$, respectively. 

The extinction conditions for a multi-type BGW are written in terms of the matrix recording the expected number of type $j$ offspring of a single type $i$ individual in one generation. The precise condition is given in~\cite[Theorem 2, Chapter V]{AthreyaNey}.
Analogous to the one dimensional case, the extinction conditions for a multi-type BGW are given in terms of a parameter $\rho$, which is the maximum eigenvalue of the mean matrix $\VEC{M}=(m_{ij})_{i,j \in [d]}$, given by $m_{ij}=\mathbb{E}(\xi^{(i,j)})$.
Given $\VEC{M}$, a strictly positive matrix, if $\rho \leq 1$, then  $\mathbb{P}(Y^{(i)}_{t} = 0 \text{ for some } t)=1$ for all $i \in [d]$. If $\rho > 1$, then  $\mathbb{P}(Y^{(i)}_{t} =0 \text{ for some } t)<1$ for all $i \in [d]$. The precise results (for the discrete-time case) can be found in~\cite[Chapter V.3]{AthreyaNey}.

\subsubsection{Phase transition for BGW process with neutral mutation}

For the BGW process with neutral mutations, the process of the total number of individuals per generation depends on the law of the underlying branching process, regardless of whether the model has finitely many or infinite alleles.  In fact, the underlying BGW tree $\mathcal{T}$ will be finite almost surely if $m  \leq 1$, where 
\begin{equation}\label{eq: defmformutation}
		m=\mathbb{E}\left[\xi^{+}\right]
\end{equation}
and $\xi^+$ is given by~\eqref{eq:offspringdist1}.  If $m >1$, $\mathcal{T}$ has a positive probability of being infinite. In this case, the relevant question is whether there exists an infinite subtree of individuals of the same type, and whether such a subtree includes the root.  We will address these matters later in Subsection \ref{subsec: phasetran mutation}.

\section{A glance at Bernoulli percolation} \label{sec:percolation}

In this section, we will present a quick overview of Bernoulli percolation.

Let us recall our graph notations introduced in Section~\ref{sec:graphs}, and let $G = (V,E)$ be a finite graph. In the sequel, it will be helpful to consider a spanning subgraph $G'$ of  $G$ as an element of $\mathcal{S} \coloneqq \{ 0, 1 \}^E = \{ s : E \rightarrow \{ 0,1 \} \} $. Observe that each $s \in \mathcal{S}$ is uniquely identified by its subset of edges $ E_{s} \coloneqq \{ e \in E \: : \:   s(e) = 1 \} $. We will refer to the edges in $E_{s}$ as \emph{open} in the configuration defined by $s$. Hence $s $ is in one-to-one correspondence with the subgraph $ (V, E_{s})$. 
With this correspondence in mind, it is convenient to refer to subgraphs as elements of $\mathcal{S}$. We will describe a random spanning subgraph as a probability distribution over the measurable space $(\mathcal{S}, \mathcal{P} (\mathcal{S}))$, where $\mathcal{P} (\mathcal{S})$ is the power set of $\mathcal{S}$.

\subsection{Bernoulli percolation}\label{sec:BerPer}

\emph{Bernoulli bond percolation} on $G = (V, E)$ with  parameter  $p \in [0,1]$ is defined by the probability measure $\mathbf{P}^{G}_p$ on the measurable space $(\mathcal{S}, \mathcal{P} (\mathcal{S}))$ given by
\begin{equation} \label{eq:perco-measure}
	\mathbf{P}^{G}_{p} (s) =  \prod_{ e \in E} p^{s (e)} (1 - p)^{1 - s (e)}, \qquad s \in \mathcal{S}.
\end{equation}
For simplicity, we will often refer to the Bernoulli bond percolation as \emph{Bernoulli percolation}.

By the correspondence established above between $\mathcal{S}$ and spanning subgraphs of $G$, we can define a random spanning subgraph $\mathcal{O}$ so that its edge set is $E_s$ with probability $\mathbf{P}^G_{p} (s)$. 
In words, the subgraph $\mathcal{O}$ is given by the following rule for each edge $e \in E$:
\begin{align} \label{eq:PercolationRule}
	&\text{the edge } e \text{ is open with probability } p \text{ and} \\
	&\text{the edge } e \text{ is closed  with probability } 1-p,  \nonumber
\end{align}
independently of other edges. We thus have that 
$E_s = \{ e \in E \: : \: e  \text{ is open} \} $. For each edge $e \in E$, the marginal  measure of  $e$ in $\mathbf{P}^G_{p}$ corresponds to a Bernoulli distribution with parameter $p$. 

The rule in~\eqref{eq:PercolationRule} can also be applied to the set of vertices $V$ 
(instead of the set of edges), which defines \emph{Bernoulli site percolation}. In this survey, we mainly focus on Bernoulli bond percolation, but Bernoulli site percolation will appear occasionally, and we will refer to it by its full name.

From the perspective of percolation theory, we are interested in the geometry of the random spanning subgraph $\mathcal{O}$. In a discrete setting, questions about the geometry of a subgraph first concern connectivity properties, including the number of connected components and their size (i.e., their number of vertices). Recall that a subgraph $G' = (V', E')$ is \emph{connected} if for any two different $x, y \in V'$, there exists a path of vertices in $G'$ between $x$ and $y$. Given a percolation configuration $\mathcal{O}$, a (random) connected subgraph $\mathcal{C}$ of open edges is called a \emph{cluster}. See Figure \ref{fig6}.

\tikzstyle{vertex}=[circle,draw,line width=1pt, color=black, fill=white,minimum size=20pt,inner sep=0pt]
\tikzstyle{selected vertex green} = [vertex, draw,color=green,line width=2pt,fill=white, text=black]
\tikzstyle{selected vertex blue} = [vertex, draw,line width=2pt, color=blue,fill=white, text=black]
\tikzstyle{selected vertex cyan} = [vertex, draw,line width=2pt,color=cyan,fill=white, text=black]
\tikzstyle{edge} = [draw,thick,-]
\tikzstyle{edgeb} = [draw,-]
\tikzstyle{weight} = [font=\small]
\tikzstyle{selected edge} = [draw,line width=5pt,-,red!50]
\tikzstyle{selected edge blue} = [draw,line width=5pt,-,blue!50]
\tikzstyle{selected edge green} = [draw,line width=5pt,-,green!50]
\tikzstyle{selected edge cyan} = [draw,line width=5pt,-,cyan!50]
\pgfdeclarelayer{background}
\pgfsetlayers{background,main}

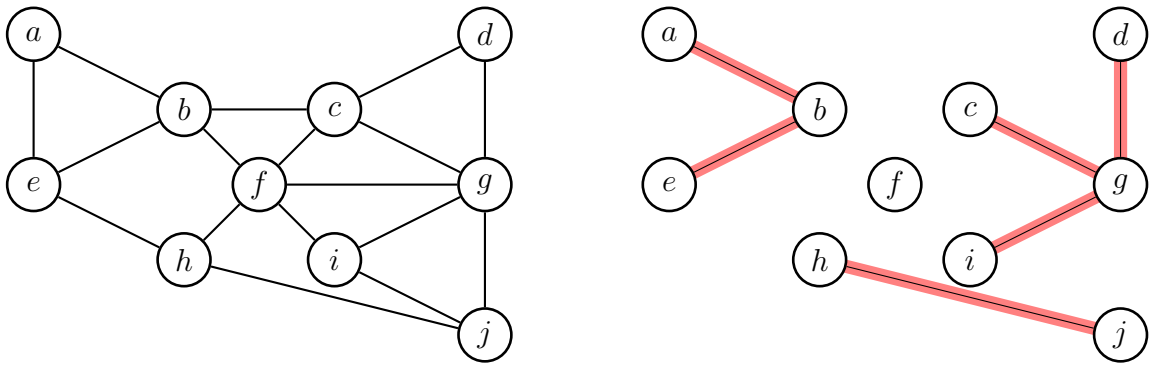
\begin{figure}[h]
	\centering
	\begin{subfigure}{.5\textwidth}
		\centering
		\begin{tikzpicture}[scale=1, auto,swap]
			\foreach \pos/\name in {{(0,2)/a}, {(2,1)/b}, {(4,1)/c}, {(6,2)/d},
				{(0,0)/e}, {(3,0)/f}, {(6,0)/g},
				{(2,-1)/h}, {(4,-1)/i}, {(6,-2)/j}}
			\node[vertex] (\name) at \pos {$\name$};
			\foreach \source/ \dest /\weight in {b/a/, c/b/, d/c/,
				e/a/, e/b/, f/b/, f/c/, 
				g/c/, g/d/, g/f/,
				h/e/, h/f/,
				i/f/, i/g/,
				j/g/, j/h/, j/i/}
			\path[edge] (\source) -- node[weight] {$\weight$} (\dest);
		\end{tikzpicture}
	\end{subfigure}
	\begin{subfigure}{.45\textwidth}
		\centering
		\begin{tikzpicture}[scale=1, auto,swap]
			\foreach \pos/\name in {{(0,2)/a}, {(2,1)/b}, {(4,1)/c}, {(6,2)/d},
				{(0,0)/e}, {(3,0)/f}, {(6,0)/g},
				{(2,-1)/h}, {(4,-1)/i}, {(6,-2)/j}}
			\node[vertex] (\name) at \pos {$\name$};
			\foreach \source/ \dest /\weight in {b/a/,g/d/,j/h/,b/e/,g/i/,g/c/}
			\path[edgeb] (\source) -- node[weight] {$\weight$} (\dest);
			\begin{pgfonlayer}{background}
				\foreach \source / \dest in {b/a/,g/d/,j/h,b/e/,g/i/,g/c/}
				\path[selected edge] (\source.center) -- (\dest.center);
			\end{pgfonlayer}
		\end{tikzpicture}
	\end{subfigure}
\caption{On the left-hand side, we have a finite connected graph $G$.  The figure on the right-hand side illustrates a realization of Bernoulli percolation with parameter $p$. The red edges correspond to the set of open edges $E_s$. The clusters of $E_s$ are $\{ a,b,e\}$, $\{ c,d,g,i\}$, $\{ f\}$ and $\{ h,j\}$. The closed edges of $G$ have been deleted.}
\label{fig6}
\end{figure}

\subsubsection{Phase transition for Bernoulli percolation on finite graphs}

A remarkable property of Bernoulli percolation on a graph $G$ is that it undergoes a phase transition in the size of its largest cluster as the parameter  $p$ varies. 
A cluster is called \emph{giant} 
if it contains a constant positive proportion of the vertices in $G$.
The \emph{supercritical} 
phase of Bernoulli percolation corresponds to the existence of a giant cluster with probability greater than a given constant. We provide quantitative definitions below. 
However, we remark that there are several ways to define the percolation phase transition for finite graphs, and these definitions are not necessarily equivalent.

Let $G$ be a finite, connected, transitive graph.
Fix $r \in (0,1)$ and $ \varepsilon > 0 $.
A cluster $K$ of a percolation configuration $\mathcal{O}$ on $G$ is $r$-giant if 
\begin{equation} \label{eq:densityR}
	 \vert K \vert \geq r \vert V \vert ,
\end{equation}
where, in a slight abuse of notation, $\vert \cdot \vert$ denotes the  cardinality of the vertex set of $K$. Hence,~\eqref{eq:densityR} indicates that the \emph{density} of the cluster $K$ in the graph $G$ is at least $r$.

Let $ (K_n)_{n \geq 1}$ be the clusters of the percolation configuration $ \mathcal{O} $ ranked in decreasing size order, so that  $K_1$ is the largest cluster.
Following the definition of Easo and Hutchcroft in~\cite{easo2024supercritical}, for each $\varepsilon > 0$, the percolation parameter $p$ is $\varepsilon$-supercritical for $G$ if  $\vert V \vert \geq 2 \varepsilon^{-3} $ (meaning that the graph is large enough) and
\[
	\mathbf{P}^{G}_{(1 - \varepsilon)p } \left(  \vert K_1 \vert  \geq  \varepsilon \vert V \vert \right) \geq \varepsilon.
\]

For an asymptotic definition of the percolation phase transition on finite graphs, we follow Bollobás, Borgs, Chayes, and Riordan~\cite{bollobas2010}. Consider a sequence of finite graphs $(G_n)_{n \geq 1} = ((V_n, E_n))_{n \geq 1}$   such that $ \vert V_n \vert \rightarrow \infty $. The sequence $ (G_n)_{n \geq 1} $ has a percolation phase transition if there exists a sequence of parameters $(p_n)_{n \geq 1}$ such that for every $\varepsilon > 0$
\begin{align*}
	\lim_{n \to \infty} \mathbf{P}_{(1 + \varepsilon)p_n}^{G_n} \left(  \vert K_1 \vert \geq \alpha \vert V \vert \right) &= 1 \qquad \text{ for some } \alpha > 0, \text{ and } \\
	\lim_{n \to \infty} \mathbf{P}_{(1 - \varepsilon)p_n}^{G_n} \left(  \vert K_1 \vert \geq \alpha \vert V \vert \right) &= 0 \qquad \text{ for all } \alpha > 0.
\end{align*}
This means that, in the supercritical regime, there is asymptotically almost surely a giant component.
The existence of a critical percolation parameter is not immediate. A characterization for the existence of a percolation phase transition for sequences of finite transitive graphs is given in~\cite{easo2023}.

\subsubsection{Bernoulli percolation on infinite graphs}

The theory of phase transitions in Bernoulli percolation was originally formulated and developed for infinite graphs. Here, we briefly review this setting and refer to~\cite{GrimmettPercolation} for a detailed account of the subject.

Let $G = (V,E)$ be an infinite, connected, and transitive graph. 
Recall that, in this work, all infinite graphs are locally finite as well. 
Consider the space of subgraphs $\mathcal{S} = \{ 0,1\}^E$. 
In this space, a \emph{cylinder event} $B$ is defined by a finite set $F = \{ e_1, \ldots , e_j \} \subset E$ and values $ a_1, \ldots , a_j \in \{ 0,1 \} $ such that 
\[ 
	B = \{ (s (e) )_{e \in E} \in \mathcal{S} \: : \: (s (e_1), \ldots , s (e_j)) = (a_1, \ldots , a_j) \}.
\]
We endow $\mathcal{S}$ with the product $\sigma$-algebra $\mathcal{F}$ generated by cylinder events. 
The Bernoulli bond percolation measure with parameter $p \in [0,1]$ on $G$  is defined by 
\begin{equation} \label{eq:perco-measure-infty}
	\mathbf{P}^{G}_{p} (B) =  \prod_{ i = 1}^{j} p^{a_i} (1 - p)^{1 - a_i}, 
\end{equation}
for any cylinder event $B$ as above.

For a transitive, connected, infinite graph $G$, we define the percolation phase transition by the existence of an infinite open cluster. 
The \emph{percolation density} $\theta : [0,1] \rightarrow [0,1]$  is the function defined by the probability that $\mathcal{C}_{\rho}$, the cluster containing a fixed vertex $\rho$, is infinite:
\[
	\theta (p) = \mathbf{P}^G_p \left( \vert \mathcal{C}_{\rho} \vert = \infty \right), \qquad p \in [0,1].
\]
The \emph{percolation critical parameter} on $G$ is defined by
\[
	p_c^{G} = \sup \{ p \in [0,1] \: : \: \theta (p) = 0 \}. 
\]

\subsubsection{Critical phenomena, universality, and locality} \label{subsec:universalitycritical}

In this subsection, we overview two fundamental concepts driving research on Bernoulli percolation: \emph{universality} and \emph{locality}. This part of the survey is independent of the rest of the work and provides perspective on the research landscape in this area.

Universality refers to the shared behavior of various physical systems at and near critical points.
When two systems share this behavior, we say that they belong to the same \emph{universality class}.
One instance where this  behavior appears is in the decay of certain quantities associated with the system, where systems in the same universality class exhibit common exponents for the power-law decay of these quantities as they approach a critical point. For example, in the case of Bernoulli bond percolation on a transitive, connected, infinite graph $G$ with critical parameter $p_c$, it is expected that
\begin{equation} \label{eq:percDensity}
		\theta (p) = (p - p_c)^{\beta + o(1)}, \qquad \text{ as } p \downarrow p_c.
\end{equation}
The exponent $\beta$, common to all systems in the same universality class, is referred to as a \emph{universal critical exponent}. It is believed that the critical exponent  $\beta$ for Bernoulli percolation depends only on the spatial dimension of the underlying graph.

On infinite transitive graphs, the concept of dimension is defined through the \emph{volume growth rate}, determined by the asymptotic behavior of the volume of closed balls in the graph metric (defined above in~\eqref{eq:graphMetric}) as their radius $r \to \infty$. For instance, in the lattice graph $\mathbb{Z}^d$, there exist constants $c, C > 0$ such that 
\[
 c r^d \leq \vert B_{\mathbb{Z}^d} (0, r) \vert \leq  C r^d, \qquad \text{ for all } r \geq 1.
\]
Thus, $\mathbb{Z}^d $ has a polynomial growth rate, with volume growth dimension $d$. 

In the 1970s,  \emph{critical phenomena} were illuminated  in physics through Wilson’s work on the renormalization group, for which he received the 1982 Nobel Prize in Physics~\cite{Wilson}. One consequence of this formalism is the conformal invariance, in the scaling limit, of critical two-dimensional statistical models, such as Bernoulli percolation~\cite{BPZa,BPZb}.

A deeper understanding of conformally invariant processes in the complex plane emerged with Schramm's introduction of the Schramm-Loewner evolution (SLE), a family of stochastic processes on the plane indexed by a parameter $\kappa > 0$~\cite{Schramm00}. Using this tool, Lawler, Werner, and Schramm made substantial progress in understanding the properties of SLE and applied it as a tool to study critical phenomena and the fractal geometry of two-dimensional Brownian motion~\cite{LSW2001a,LSW2001b,LSW2002c,LSW2002b,LSW2002a,LSW2004}. For these results, Werner received the Fields Medal in 2006~\cite{WernerFields}.

A rigorous proof of the conformal invariance of the scaling limit of two-dimensional critical Bernoulli percolation was achieved by Smirnov, who proved the conjectured conformal invariance for Bernoulli site percolation on the triangular lattice~\cite{Smirnov}. This result showed that SLE  with parameter $\kappa = 6$ is the scaling limit of two-dimensional critical Bernoulli site percolation on the triangular lattice. This landmark work also led to Smirnov's receipt of the Fields Medal in 2010~\cite{SmirnovFields}. A summary of the results by Lawler, Werner, Schramm, and Smirnov, along with Kesten's previous work, appears in~\cite{SmirnovWerner}. In this work, Smirnov and Werner establish the existence and values of universal critical exponents for Bernoulli percolation on the triangular lattice, confirming that the value of the exponent in~\eqref{eq:percDensity} is  $\beta = 5/36$, as predicted in the physics literature.

Percolation theory has extended the techniques of Bernoulli percolation to  a larger class of percolation models, allowing for dependencies. One significant model in this class is the \emph{random-cluster model} (also called FK percolation), introduced by Fortuin and Kasteleyn in the 1970s~\cite{FortuinKasteleyn,FortuinII,FortuinIII}. It generalizes Bernoulli percolation and provides a unifying framework for models such as the Ising and Potts models, as well as Bernoulli percolation itself. Taking limits on the parameters, the random-cluster model includes the uniform spanning tree and the $\beta$-arboreal gas models~\cite{JacobsenSalasSokal}. For a comprehensive treatment, see~\cite{Grimmett}.

Duminil-Copin and collaborators have made significant progress in understanding the phase transitions and the critical and near-critical behavior of the random-cluster model~\cite{AizenmanDuminilCopinSidoravicius,BeffaraDuminilCopin,DuminilCopinGagnebinHarelManolescuTassion,DuminilCopinKozlowskiKrachunManolescuOulamara,DuminilCopinRaoufiTassion,DuminilSidoraviciusTassion}. This work has also led to new insights into Bernoulli percolation, simplifying the proofs of key results, such as a proof of the sharpness of the phase transition~\cite{DuminilCopinTassion2016new}. Duminil-Copin’s work earned him a Fields Medal in 2022~\cite{HugoFields}.

In Bernoulli percolation, several questions remain open. One of the biggest challenges is to understand the phase transition and critical behavior in three dimensions, such as on $\mathbb{Z}^3$.
Another rich area of research explores Bernoulli percolation on graphs beyond $\mathbb{Z}^d$, where this model has generated deep connections with other mathematical fields, including geometric group theory.

Studying Bernoulli percolation on diverse graphs has highlighted different properties of the model. Since the early days of percolation theory, numerical analysis has shown that the value of $p^G_c$ is not universal; rather, it depends on the graph $G$. In fact, whenever $p^G_c < 1$, the value of $p^G_c$ depends on the graph's \emph{local} geometry. 
This is a remarkable phenomenon, since locality is the opposite of universality.

The local geometry of a graph is measured by the \emph{local topology} (also known as the \emph{Benjamini-Schramm topology}). 
A sequence of connected, infinite, transitive graphs $(G_n)_{n \geq 1}$ converges to a  connected, infinite, transitive graph $G$ in the local topology if the following holds: for each $r > 0$, the closed balls of radius $r$ in $G_n$ are isomorphic to those in $G$ for all sufficiently large $n$.

In the space of connected, infinite, transitive graphs, let $(G_n)_{n \geq 1}$ be a sequence converging to a graph $G$ in the local topology as $n \to \infty$. If $\sup_{n \geq 1} p_c(G_n) < 1$, Schramm conjectured that $ p_c (G_n) \rightarrow p_c (G)$ as $n \to \infty$.

A key step toward understanding this conjecture was achieved in~\cite{duminil2020existence}, where the authors showed that the condition $p_c (G_n) = 1$ is equivalent to $G_n$ being one-dimensional, meaning that its volume growth satisfies $\vert B_{G_n} (v, r) \vert = O(r)$ as $r \to \infty$.

Based on~\cite{duminil2020existence}, for Schramm's locality conjecture, it is sufficient to assume that each element in the sequence $(G_n)_{n \geq 1}$ has superlinear volume growth. 

Schramm's locality conjecture was recently settled in a breakthrough work by Hutchcroft and Easo, who proved that, given a sequence of connected, infinite, transitive graphs $(G_n)_{n \geq 1}$ with superlinear volume growth, if $G_n \rightarrow G$ as $n \to \infty$ in the local topology, then $p_c (G_n) \rightarrow p_c (G)$ as $n \to \infty$~\cite{EasoHutchcroft23}.

\subsection{Bernoulli percolation on trees} \label{subsec: BpBGW}

We now turn to a specific class of graphs and begin with a brief account of Bernoulli percolation on infinite trees.
We begin by looking at Bernoulli percolation on $d$-ary trees. For $d \geq 2$, a $d$-ary 
tree is an infinite rooted tree in which each vertex has $d$ children.

\begin{proposition}  \label{prop:percolation-dtree}
		Let $T_d$ be a $d$-ary tree with $d \geq 2$ and consider Bernoulli percolation with parameter $p \in [0,1]$ on $T_d$. 
		If $\mathcal{T}_{\emptyset}$ is the open cluster of the root, then $\mathcal{T}_{\emptyset}$ follows the distribution of a BGW tree with progeny distribution $\Bin(d, p)$. 
\end{proposition}

A direct consequence of Proposition~\ref{prop:percolation-dtree} and the phase transition on the survival of a BGW tree (as explained in Subsection~\ref{subsection:extinctionBGW}) is the following.

\begin{proposition} \label{prop:critical-dtree}
	Let $T_d$ be a $d$-ary tree with $d \geq 2$. Then the critical percolation parameter is  $p_c (T_d) = 1/d $. 
\end{proposition}

A generalization of the previous proposition to general trees involves the definition of the \emph{branching number}. On an infinite tree, the branching number corresponds to the average number of branches per vertex. This notion arises from the work of Furstenberg on the Hausdorff dimension of a tree~\cite{Furstenberg}. However, its precise definition goes beyond the scope of this survey; we refer the reader to~\cite{Lyons90} or~\cite[Section 1.2]{lyons2017probability}. We would like to point out that the branching number of a $d$-ary tree is $d$. In this sense, the following theorem extends Proposition~\ref{prop:critical-dtree}.

\begin{theorem}[{\cite{Lyons90}}] \label{thm:branchingPerc}
	Let $T$ be an infinite tree with branching number $b$. Then its critical percolation parameter is  $p_c (T) = 1/b $. 
\end{theorem}

\subsection{Coloring percolation clusters at random}\label{sec:DaC}

In~\cite{Haggstrom}, Häggström  introduced  a percolation model on a graph $G$ called the {\it divide and color  model} (hereafter abbreviated as \emph{DaC}). This model generates a random coloring of the vertices of $G$. It proceeds by first performing Bernoulli bond percolation with parameter  $p$ on the graph $G$ (see Figure~\ref{fig6}) and then assigning a random color, chosen according to a prescribed probability distribution on the finite set $[d]$, to each percolation cluster. The coloring of different clusters is done independently (as illustrated in Figure~\ref{figdac}). 

In what follows, we introduce the model more precisely for the case of our interest, namely, when $G$ is a tree.

\begin{figure}[h]
	\centering
	\begin{subfigure}{.5\textwidth}
		\centering
		\begin{tikzpicture}[scale=1, auto,swap]
			\foreach \pos/\name in {{(0,2)/a}, {(2,1)/b}, {(4,1)/c}, {(6,2)/d},
				{(0,0)/e}, {(3,0)/f}, {(6,0)/g},
				{(2,-1)/h}, {(4,-1)/i}, {(6,-2)/j}}
			\node[vertex] (\name) at \pos {$\name$};
			\foreach \source/ \dest /\weight in {b/a/, c/b/, d/c/,
				e/a/, e/b/, f/b/, f/c/, 
				g/c/, g/d/, g/f/,
				h/e/, h/f/,
				i/f/, i/g/,
				j/g/, j/h/, j/i/}
			\path[edge] (\source) -- node[weight] {$\weight$} (\dest);
		\end{tikzpicture}
	\end{subfigure}
	\begin{subfigure}{{.45\textwidth}}
		\centering
		\begin{tikzpicture}[scale=1, auto,swap]
			
			\foreach \pos/\name in {{(0,2)/a}, {(2,1)/b}, {(4,1)/c}, {(6,2)/d},
				{(0,0)/e}, {(3,0)/f}, {(6,0)/g},
				{(2,-1)/h}, {(4,-1)/i}, {(6,-2)/j}}
			\node[vertex] (\name) at \pos {$\name$};
			\foreach \source/ \dest /\weight in {b/a/,g/d/,j/h/,b/e/,g/i/,g/c/}
			\path[edgeb] (\source) -- node[weight] {$\weight$} (\dest);
			
			\foreach \vertex  in {a,b,e}
			\path node[selected vertex green] at (\vertex) {$\vertex$};

			\foreach \vertex in {d,g,c,i}
			\path node[selected vertex cyan] at (\vertex) {$\vertex$};
			
			\foreach \vertex in {h,j,f}
			\path node[selected vertex blue] at (\vertex) {$\vertex$};
			
			\foreach \source/ \dest /\weight in {b/a/,g/d/,j/h/,b/e/,g/i/,g/c/}
			\path[edge] (\source) -- node[weight] {$\weight$} (\dest);

			\begin{pgfonlayer}{background}
				\foreach \source / \dest in {b/a/,b/e/}
				\path[selected edge green] (\source.center) -- (\dest.center);
				
				\foreach \source / \dest in {g/d/,g/i/,g/c/}
				\path[selected edge cyan] (\source.center) -- (\dest.center);
				
				\foreach \source / \dest in {j/h}
				\path[selected edge blue] (\source.center) -- (\dest.center);
			\end{pgfonlayer}
		\end{tikzpicture}
	\end{subfigure}
	\caption{On the left-hand side, we have a finite connected graph, while on the right-hand side, we see a realization of a  DaC percolation. This realization was obtained by applying a random coloring to the Bernoulli percolation illustrated in Figure~\ref{fig6}. 
	In this example, the vertices $ \{  f, h, j \} $ belong to the same DaC connected component, since they share the same color and are connected in the original graph.}
	\label{figdac}
\end{figure}
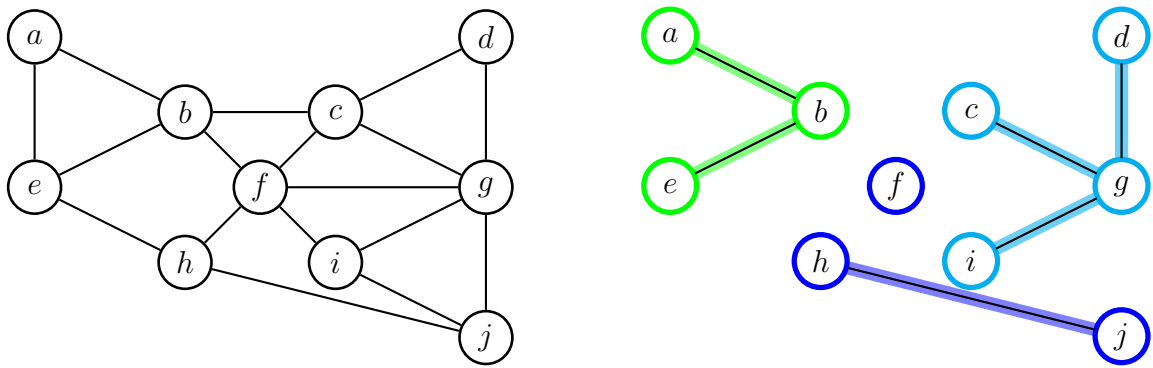

\subsubsection{Coloring percolation on trees}\label{subsubsec: dac on trees}

Let $\VEC{t}=(V,E)$  be an infinite tree, and let there be $d \geq 2$ different colors indexed by the set $[d]$. In addition to the percolation parameter $p$ used to define Bernoulli bond percolation, we consider a probability distribution $\VEC{a}=(a_1,\dots,a_d)$ on the set of colors  $[d]$. The coloring is done according to the following two-step procedure:
\begin{itemize}
	\item[] {\bf Step 1}:  Perform Bernoulli bond percolation on $\VEC{t}$ (as described in Subsection~\ref{sec:BerPer}). As a result, we obtain a random subgraph of $\VEC{t}$, denoted by $\mathcal{O}$, and we let $\mathscr{C}$ be its set of clusters. For clarity, we will refer to the elements of $\mathscr{C}$ as \emph{percolation clusters}.
	\item[] {\bf Step 2}: For each percolation cluster in $\mathscr{C}$, assign the same color to all its vertices. This color is chosen according to the distribution $\VEC{a}$, independently for different clusters. The resulting coloring $\Pi$  is defined by a probability measure on $[d]^V$, and the law of $\Pi$ is denoted by $\mathbb{P}_{p, d, \VEC{a}}^{\VEC{t}}$. 
\end{itemize}

Once the coloring $\Pi$ is obtained, the \emph{DaC connected components} are defined as the sets of connected vertices (in the graph $\VEC{t}$) that share the same color.

Throughout this paper, we focus on the DaC model with just two colors.
Hence, we use the notation $\mathbb{P}_{p, 2, a_1}$, where $a_1$ is the probability of coloring a vertex on the graph with color $1$. Note that $1-a_1$
is the probability of coloring the vertex with color 2 (different from color $1$). The development in~\cite{Haggstrom} includes the study of the DaC model for infinite trees and the corresponding phase transition results for $d=2$.
A reason for considering only two colors is that the question of whether the DaC measure $\mathbb{P}_{p, d, \VEC{a}}^{\VEC{t}}$ produces an infinite connected component of color $i\in [d]$ with $d\geq 2$ can be answered by examining the measure $\mathbb{P}_{p, 2, a_i}^{\VEC{t}}$, since $1-a_i$ is the probability of not coloring the vertex with color 1.
We summarize the phase transition results of~\cite{Haggstrom} in the following proposition.

\begin{proposition}[{\cite[Propositions 2.7 and 2.8]{Haggstrom}}] \label{prop:Haggstrom}   
	Let $\Pi$ be the resulting coloring obtained by applying the DaC model to an infinite tree $\VEC{t}$ with parameters  $p \in [0,1]$, $d=2$ and $ \VEC{a}=(a_{1}, a_2)$. Then
	\[
    \mathbb{P}_{p, 2, a_1}^{\VEC{t}}
    \left( \Pi \text{ contains an infinite DaC connected component of color } 1 \right) > 0.
  \] 
  if and only if
  \[
		\mathbf{P}_{\widetilde{p}}^{\VEC{t}}
    \left( \mathcal{O} \text{ contains an infinite cluster} \right) >0
  \]
   where $ \widetilde{p} = 1 - (1-p)(1 - a_{1})$, and $\mathbf{P}_{\widetilde{p}}^{\VEC{t}}$ is the probability measure defined in~\eqref{eq:perco-measure-infty}.
\end{proposition}

According to~\cite[Proof of Proposition 2.7]{Haggstrom}, the idea behind the proof of Proposition~\ref{prop:Haggstrom} is that a DaC connected component (of a given color and containing a given vertex) has the same distribution as a cluster in Bernoulli site percolation (see Figure~\ref{fig:8}). However, one must be careful, as a configuration determined by the DaC model on  $\VEC{t}$ is not equivalent in distribution to a configuration given by Bernoulli site percolation on the vertices of $\VEC{t}$ (see \cite[Remark p.227]{Haggstrom}).

Indeed,  the proof of Proposition~\ref{prop:Haggstrom} examines each DaC connected component using an exploration argument. Without loss of generality, let us explain this by analyzing a genealogical tree, focusing on the root $\emptyset$ and its children.
Note that the offspring of the root have two possibilities for inheriting the same color as their mother:
\begin{itemize}
	\item [1)] 
	A child of the root belongs to the same percolation cluster as its mother; this occurs with probability $p$.
	\item [2)] 
	A child of the root belongs to a different percolation cluster than its mother, but this connected component still chooses the mother's color. This event has probability  $(1-p)a_{\emptyset} $, where $a_{\emptyset}$ 
	is the probability of coloring a cluster with the same color as the root.
\end{itemize}
Therefore, a child of the root has the same color as its mother with probability 
\[
	p + (1-p)a_{\emptyset} = 1 - (1-p)(1 - a_{\emptyset}).
\]

Let $\widetilde{\VEC{t}}_{\emptyset}$ denote the subtree in which every individual has the same color as the root and no ancestors of another color. Thanks to the previous observation, we can directly conclude the following result.

\begin{proposition}
Consider the DaC model on an infinite tree $\VEC{t}$ with parameters $p \in [0,1]$, $d=2$, and let $a_{\emptyset}$ be the probability of coloring a percolation cluster with the same color as the root. We have that
  \[
    \mathbb{P}_{p, d, a_{\emptyset}}^{\VEC{t}}
    \left(  \widetilde{\VEC{t}}_\emptyset \text{ has infinite size} \right) > 0
  \] 
  if and only if
  \[  \mathbf{P}_{\widetilde{p}}^{\VEC{t}}
    \left( \mathcal{C}_{\emptyset} \text{ has infinite size} \right) >0,
  \]
  where $ \widetilde{p} = 1 - (1-p)(1 - a_{\emptyset})$, and $\mathcal{C}_{\emptyset}$ is the root cluster obtained by applying Bernoulli percolation with parameter $\widetilde{p}$ on $\VEC{t}$.
\end{proposition}

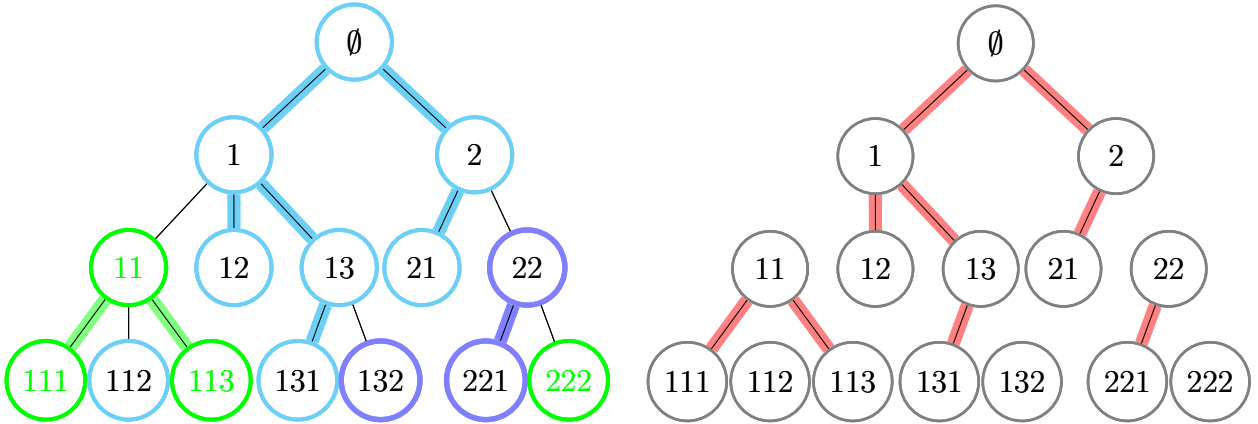
\begin{figure}[h]
	\centering
	\begin{subfigure}{.5\textwidth}
		\centering
		\begin{tikzpicture}
	[
	level 1/.style={sibling distance=32mm},
	level 2/.style={sibling distance=14mm},
	level 3/.style={sibling distance=11mm},
	every node/.style={circle, draw=cyan!50,line width=1.5pt, minimum size = 1 cm},
	bluen/.style={circle, draw=blue!50,line width=2pt, minimum size = 1 cm},
	greenn/.style={circle, draw=green!50,line width=2pt, minimum size = 1 cm},
	emph/.style={edge from parent/.style={draw,line width=1.5pt,-,black!50}},
	rede/.style={edge from parent/.style={draw,line width=1.5pt,-,red!50}},
	bluee/.style={edge from parent/.style={draw,line width=5pt,-,blue!50}},
	greene/.style={edge from parent/.style={draw,line width=5pt,-,green!50}},
	cyane/.style={edge from parent/.style={draw,line width=5pt,-,cyan!50}},
	norm/.style={edge from parent/.style={solid,black,thin,draw}},
	noe/.style={edge from parent/.style={}},
	]
	\node {$\emptyset$}
		child[norm] {node {1}
			child[norm] {node[green] {11}
			child[norm] {node[green] {111}}
				child {node {112}}
				child[norm] {node[green] {113}}}
			child[norm] {node {12}}
			child {node {13}
		    child {node {131}}
	        child[norm] {node[bluen] {132}}}}
		child[norm] {node {2}
		child[norm] {node {21}}
	    child[norm] {node[bluen] {22}
        child[norm] {node[bluen] {221}}
        child {node[green] {222}}}}
		;
	\begin{pgfonlayer}{background}
	\node {$\emptyset$}
		child[cyane] {node {1}
			child[norm] {node[green] {11}
			child[greene] {node[green] {111}}
				child {node {112}}
				child[greene] {node[green] {113}}}
			child[cyane] {node {12}}
			child {node {13}
		    child {node {131}}
	        child[norm] {node[bluen] {132}}}}
		child[cyane] {node {2}
		child[cyane] {node {21}}
	    child[norm] {node[bluen] {22}
        child[bluee] {node[bluen] {221}}
        child {node[green] {222}}}}
		;
	\end{pgfonlayer}
\end{tikzpicture}
	\end{subfigure}
	\begin{subfigure}{{.45\textwidth}}
		\centering
         \begin{tikzpicture}			[
		level 1/.style={sibling distance=32mm},
		level 2/.style={sibling distance=14mm},
		level 3/.style={sibling distance=11mm},
			every node/.style={circle, draw=black!50,line width=1pt, minimum size = 1 cm},
			bluen/.style={circle, draw=blue!50,line width=2pt, minimum size = 1 cm},
			greenn/.style={circle, draw=green!50,line width=2pt, minimum size = 1 cm},
			emph/.style={edge from parent/.style={draw,line width=1.5pt,-,black!50}},
			rede/.style={edge from parent/.style={draw,line width=5pt,-,red!50}},
			bluee/.style={edge from parent/.style={draw,line width=5pt,-,blue!50}},
			greene/.style={edge from parent/.style={draw,line width=5pt,-,green!50}},
			cyane/.style={edge from parent/.style={draw,line width=5pt,-,cyan!50}},
			norm/.style={edge from parent/.style={solid,black,thin,draw}},
			noe/.style={edge from parent/.style={}},
			]
			\node {$\emptyset$}
			child[norm] {node {1}
				child[noe] {node {11}
					child[norm] {node {111}}
					child[noe] {node {112}}
					child[norm] {node {113}}}
				child {node {12}}
				child {node {13}
					child {node {131}}
					child[noe] {node {132}}}}
			child[norm] {node {2}
				child {node {21}}
				child[noe] {node {22}
					child[norm] {node {221}}
					child {node {222}}}}
			;
			\begin{pgfonlayer}{background}
			\node {$\emptyset$}
			child[rede] {node {1}
				child[noe] {node {11}
					child[rede] {node {111}}
					child[noe] {node {112}}
					child[rede] {node {113}}}
				child {node {12}}
				child {node {13}
					child {node {131}}
					child[noe] {node {132}}}}
			child[rede] {node {2}
				child {node {21}}
				child[noe] {node {22}
					child[rede] {node {221}}
					child {node {222}}}}
			;
			\end{pgfonlayer}
		\end{tikzpicture}
	\end{subfigure}
	\caption{The figure on the left shows a realization of the \emph{DaC} model on a finite tree with parameter $p$ and $d=3$. The figure on the right illustrates  the corresponding percolation, where the connected components are distinguished by color.} \label{fig:8}
\end{figure}

\section{Connections between percolation models and branching structures} \label{sec:connections}

This section is devoted to providing a representation of the BGW branching process with neutral mutations, as presented in Section~\ref{sec:GWmut}, in terms of Bernoulli percolation and the DaC model. 

In~\cite{Bertoin09}, Bertoin introduced the model of neutral mutations on infinite alleles and implicitly used percolation in BGW trees. This connection is made explicit in~\cite{BertoinUB}. In Subsection~\ref{subsec: relatbtwmodels}, we generalize this relationship in two ways: from the point of view of population genetics (BGW process with multi-allelic neutral mutations) and from the point of view of percolation theory (DaC percolation model).
First, we will discuss the relation between BGW processes with infinite-allele neutral mutations and Bernoulli percolation introduced by Bertoin~\cite{Bertoin09} and further studied in~\cite{Bertoin10}. Next, we will introduce the relation between DaC percolation and the MIM model. Finally, at the end of Subsection~\ref{subsec: relatbtwmodels}, we will present the \emph{restricted DaC percolation model}, related to the MDM model.

\subsection{From Bernoulli percolation to neutral mutations and back}
\label{subsec: relatbtwmodels}
\subsubsection{Bernoulli percolation on BGW trees} \label{subsub:PercolationNeutralMutations}

In Subsection~\ref{sec:BerPer}, we defined Bernoulli percolation with percolation parameter $p$ on a deterministic graph. 
In this section, we apply the same procedure to a random tree, in particular, to a BGW tree. Note that by taking
$p=1-r$, where $r$ is the mutation parameter in a BGW process with infinite alleles (neutral mutation), we obtain a unique correspondence between both models. Roughly speaking, a mutation corresponds to closing an edge.
In terms of population genetics, opening an edge can be interpreted as the mother and child sharing the same type. Indeed, by the branching property, it is sufficient to analyze the root $\emptyset$ and its children. Conditioned on $\{\xi^{(+) }=v \}$, the distribution of the number of clones is given by  
\begin{equation*}
	\mathbb{P} \left( \xi^{(c)}=u \vert \xi^{(+) }=v  \right)= \binom{v }{u}(1-r)^{u}r^{v-u},
\end{equation*}
i.e.,  conditioned on  $\{\xi^{(+)} =v \} $, we have $\xi^{(c)} \mid_ {\{\xi^{(+) }=v \}} \sim$Binomial$( v, 1-r)$. Therefore,  for each child, its  marginal conditional measure corresponds to  a Bernoulli distribution with parameter $1-r$. 
This corresponds to the marginal of the measure $\mathbf{P}^G_{1-r}$ defined in~\eqref{eq:perco-measure} for finite graphs. An example of this comparison in shown in Figure~\ref{fig:9}.

\begin{figure}[h]
	\centering
	\begin{subfigure}{{.45\textwidth}}
		\centering
		\begin{tikzpicture}
	  	[
      level 1/.style={sibling distance=45mm},
      level 2/.style={sibling distance=30mm},
      level 3/.style={sibling distance=12mm},
      level 4/.style={sibling distance=10mm},]
      \node[circle, draw,inner sep=6](root) { $\emptyset$}
      child {node[circle,draw,inner sep=6] {1}
      	child {node[circle,draw,inner sep=5] {11}
      		child {node[regular polygon,
      			regular polygon sides=6,draw,inner sep=0.2] {111} edge from parent node {\textcolor{red}{{\sf X}}}}
      	}
      	child {node[regular polygon,
      		regular polygon sides=5,draw,inner sep=3] {12} 
      		child {node[regular polygon,
      			regular polygon sides=5,draw,inner sep=0.2] {121}}edge from parent node {\textcolor{red}{{\sf X}}}} }
      child {node[regular polygon,
      	regular polygon sides=4,draw,inner sep=5] {2}
      	child {node[regular polygon,
      		regular polygon sides=4,draw,inner sep=3] {21}
      		child {node[regular polygon,
      			regular polygon sides=7,draw,inner sep=2] {212}edge from parent node {\textcolor{red}{{\sf X}}}}}edge from parent node {\textcolor{red}{{\sf X}}}};
    \end{tikzpicture}
	\end{subfigure}
	\centering
	\begin{subfigure}{{.45\textwidth}}
		\centering
			\begin{tikzpicture}
			[
			level 1/.style={sibling distance=45mm},
			level 2/.style={sibling distance=30mm},
			level 3/.style={sibling distance=12mm},
			level 4/.style={sibling distance=10mm},
			every node/.style={circle, draw,line width=1pt, minimum size = 1 cm},
			bluen/.style={circle, draw=blue!50,line width=2pt, minimum size = 1 cm},
			greenn/.style={circle, draw=green!50,line width=2pt, minimum size = 1 cm},
			cyann/.style={circle, draw=cyan!50,line width=2pt, minimum size = 1 cm},
			tealn/.style={circle, draw=teal!50,line width=2pt, minimum size = 1 cm},
			oliven/.style={circle, draw=olive!50,line width=2pt, minimum size = 1 cm},
			violetn/.style={circle, draw=violet!50,line width=2pt, minimum size = 1 cm},
			emph/.style={edge from parent/.style={draw,line width=5pt,-,darkgray!50}},
			rede/.style={edge from parent/.style={draw,line width=5pt,-,red!50}},
			bluee/.style={edge from parent/.style={draw,line width=5pt,-,blue!50}},
			greene/.style={edge from parent/.style={draw,line width=5pt,-,green!50}},
			cyane/.style={edge from parent/.style={draw,line width=5pt,-,cyan!50}},
			norm/.style={edge from parent/.style={solid,black,thin,draw}},
			noe/.style={edge from parent/.style={}},
			]
			\node[circle, draw=black!50,inner sep=6](root) { $\emptyset$}
			child[norm] {node[circle,draw=black!50,inner sep=6] {1}
				child[norm] {node[circle,draw=black!50,inner sep=5] {11}
					child[noe] {node[regular polygon,
						regular polygon sides=6,draw=black!50,inner sep=0.2] {111}}
				}
				child[noe] {node[regular polygon,
					regular polygon sides=5,draw=black!50,inner sep=3] {12} 
					child[norm] {node[regular polygon,
						regular polygon sides=5,draw=black!50,inner sep=0.2] {121}}} }
			child[noe] {node[regular polygon,
				regular polygon sides=4,draw=black!50,inner sep=5] {2}
				child[norm] {node[regular polygon,
					regular polygon sides=4,draw=black!50,inner sep=3] {21}
					child[noe] {node[regular polygon,
						regular polygon sides=7,draw=black!50,inner sep=2] {212}}}};
			\begin{pgfonlayer}{background}
			\node[circle, draw=black!50,inner sep=6](root) { $\emptyset$}
			child[rede] {node[circle,draw=black!50,inner sep=6] {1}
				child[rede] {node[circle,draw=black!50,inner sep=5] {11}
					child[noe] {node[regular polygon,
						regular polygon sides=6,draw=black!50,inner sep=0.2] {111}}
				}
				child[noe] {node[regular polygon,
					regular polygon sides=5,draw=black!50,inner sep=3] {12} 
					child[rede] {node[regular polygon,
						regular polygon sides=5,draw=black!50,inner sep=0.2] {121}}} }
			child[noe] {node[regular polygon,
				regular polygon sides=4,draw=black!50,inner sep=5] {2}
				child[rede] {node[regular polygon,
					regular polygon sides=4,draw=black!50,inner sep=3] {21}
					child[noe] {node[regular polygon,
						regular polygon sides=7,draw=black!50,inner sep=2] {212}}}};
			\end{pgfonlayer}
		\end{tikzpicture}
	\end{subfigure}
	\caption{Comparison between the infinite alleles model and Bernoulli percolation.} \label{fig:9} 
\end{figure}
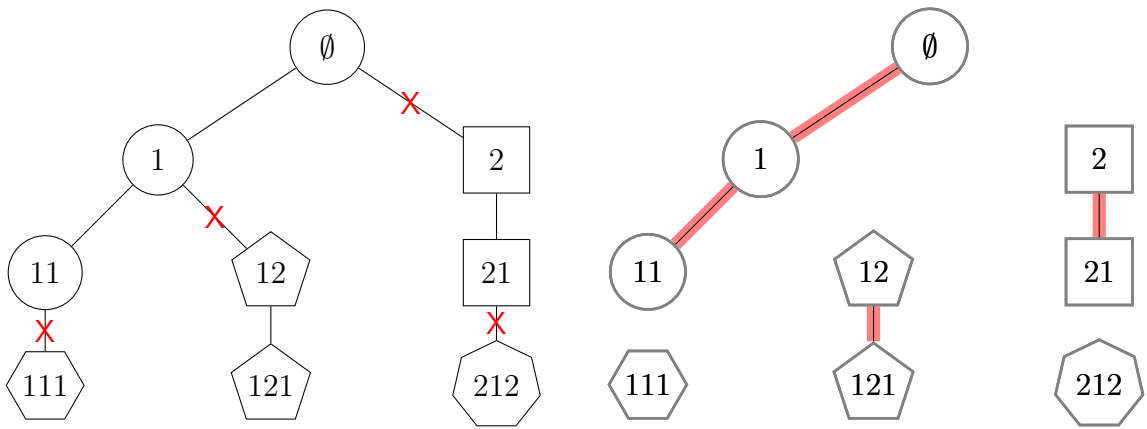

\subsubsection{Divide and color percolation on BGW trees}  \label{subsub:DaC-MIM}

In Subsection~\ref{sec:DaC}, we defined the DaC model on an infinite tree $\VEC{t}$ with parameters $p \in [0,1]$, $d\geq2$, and $\VEC{a}=(a_1,\dots,a_d)$. We can apply the same model to a BGW tree.
 As before, we can establish a correspondence with the mother-independent mutation model defined in Subsection~\ref{subsub: MDM}. To do this, we take
 $p=1-r$, where $r$ is the mutation probability. We also let $d$ be the number of types and set  $a_i= 1/d$ for all $i\in[d]$, since each mutant child chooses its color uniformly. 

Again, by the branching property, it is enough to analyze the root $\emptyset$ and its children. Assume that the type of the root is $i$.
Let $ \VEC{u} = (u_1, \ldots, u_d) \in \mathbb{N}_0^d $ and recall that  $|\VEC{u}| = \sum_{j=1}^d u_j$.
Following equations~\eqref{eq:offspringdist1} and~\eqref{eq:mothindproba}, conditioned on $\{\xi^{(+, i)}= |\VEC{u}|\}$,  the probability that the root has $\boldsymbol{\xi}^i \coloneqq (\xi^{(i,1)},\xi^{(i,2)},\dots,\xi^{(i,d)})=(u_1,u_2,..,u_d)$ children is given by   
\begin{equation}
	\mathbb{P}\left(\boldsymbol{\xi}^i= \VEC{u} ) \vert \xi^{(+, i)}= |\VEC{u}|\right)=
	\sum_{k = 1}^{u_i}
	\binom{|\VEC{u}|}{k} (1-r)^{k}
	r^{|\VEC{u}|-k}  
	\binom{|\VEC{u}|-k}{u_1, \dots , u_i - k , \dots, u_d}
	\left(\dfrac{1}{d}\right)^{ |\VEC{u}| - k }.
\end{equation}
This measure coincides with the DaC measure $\mathbb{P}_{p, d, \VEC{a}}$ applied to the root and its children at the first generation (see Figures~\ref{fig: mthind mutation} and~\ref{fig:10}). 

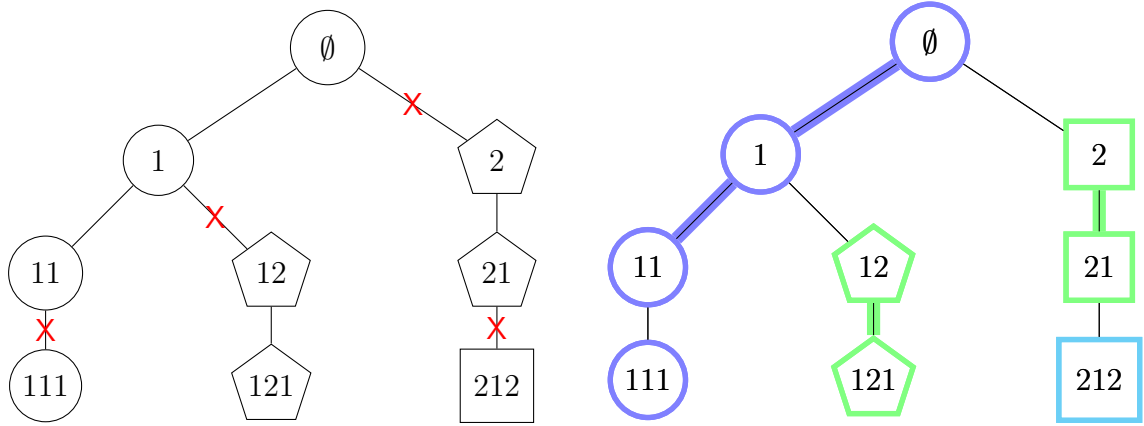
\begin{figure}[h]
	\centering
	\begin{subfigure}{{.45\textwidth}}
		\centering
		\begin{tikzpicture}
			[
			level 1/.style={sibling distance=45mm},
			level 2/.style={sibling distance=30mm},
			level 3/.style={sibling distance=12mm},
			level 4/.style={sibling distance=10mm},]
			\node[circle, draw,inner sep=6](root) { $\emptyset$}
			child {node[circle,draw,inner sep=6] {1}
				child {node[circle,draw,inner sep=5] {11}
					child {node[circle,draw,inner sep=3] {111} edge from parent node {\textcolor{red}{{\sf X}}}}
				}
				child {node[regular polygon,
					regular polygon sides=5,draw,inner sep=3] {12} 
					child {node[regular polygon,
						regular polygon sides=5,draw,inner sep=0.2] {121}}edge from parent node {\textcolor{red}{{\sf X}}}} }
			child {node[regular polygon,
				regular polygon sides=5,draw,inner sep=5] {2}
				child {node[regular polygon,
					regular polygon sides=5,draw,inner sep=3] {21}
					child {node[regular polygon,
						regular polygon sides=4,draw,inner sep=1] {212}edge from parent node {\textcolor{red}{{\sf X}}}}}edge from parent node {\textcolor{red}{{\sf X}}}};
		\end{tikzpicture}
	\end{subfigure}
	\centering
	\begin{subfigure}{{.45\textwidth}}
		\centering
		\begin{tikzpicture}
			[
			level 1/.style={sibling distance=45mm},
			level 2/.style={sibling distance=30mm},
			level 3/.style={sibling distance=12mm},
			level 4/.style={sibling distance=10mm},
			every node/.style={circle, draw=darkgray!50,line width=2pt, minimum size = 1 cm},
			bluen/.style={circle, draw=blue!50,line width=2pt, minimum size = 1 cm},
			greenn/.style={circle, draw=green!50,line width=2pt, minimum size = 1 cm},
			cyann/.style={circle, draw=cyan!50,line width=2pt, minimum size = 1 cm},
			tealn/.style={circle, draw=teal!50,line width=2pt, minimum size = 1 cm},
			oliven/.style={circle, draw=olive!50,line width=2pt, minimum size = 1 cm},
			violetn/.style={circle, draw=violet!50,line width=2pt, minimum size = 1 cm},
			emph/.style={edge from parent/.style={draw,line width=5pt,-,darkgray!50}},
			rede/.style={edge from parent/.style={draw,line width=5pt,-,red!50}},
			bluee/.style={edge from parent/.style={draw,line width=5pt,-,blue!50}},
			greene/.style={edge from parent/.style={draw,line width=5pt,-,green!50}},
			cyane/.style={edge from parent/.style={draw,line width=5pt,-,cyan!50}},
			norm/.style={edge from parent/.style={solid,black,thin,draw}}
			]
			\node[circle, draw=blue!50,inner sep=6](root) { $\emptyset$}
			child[norm] {node[circle,draw=blue!50,inner sep=6] {1}
				child {node[circle,draw=blue!50,inner sep=5] {11}
					child[norm] {node[circle,draw=blue!50,inner sep=0.2] {111}}
				}
				child[norm] {node[regular polygon,
					regular polygon sides=5,draw=green!50,inner sep=3] {12} 
					child[norm] {node[regular polygon,
						regular polygon sides=5,draw=green!50,inner sep=0.2] {121}}} }
			child {node[regular polygon,
				regular polygon sides=4,draw=green!50,inner sep=5] {2}
				child[norm] {node[regular polygon,
					regular polygon sides=4,draw=green!50,inner sep=3] {21}
					child[norm] {node[regular polygon,
						regular polygon sides=4,draw=cyan!50,inner sep=2] {212}}}};
			\begin{pgfonlayer}{background}
			\node[circle, draw=blue!50,inner sep=6](root) { $\emptyset$}
			child[bluee] {node[circle,draw=blue!50,inner sep=6] {1}
				child {node[circle,draw=blue!50,inner sep=5] {11}
					child[norm] {node[circle,draw=blue!50,inner sep=0.2] {111}}
				}
				child[norm] {node[regular polygon,
					regular polygon sides=5,draw=green!50,inner sep=3] {12} 
					child[greene] {node[regular polygon,
						regular polygon sides=5,draw=green!50,inner sep=0.2] {121}}} }
			child {node[regular polygon,
				regular polygon sides=4,draw=green!50,inner sep=5] {2}
				child[greene] {node[regular polygon,
					regular polygon sides=4,draw=green!50,inner sep=3] {21}
					child[norm] {node[regular polygon,
						regular polygon sides=4,draw=cyan!50,inner sep=2] {212}}}};
			\end{pgfonlayer}
		\end{tikzpicture}
	\end{subfigure}
	\caption{Comparison between the MIM model and DaC percolation.} \label{fig:10}
\end{figure}

\subsubsection{Restricted divide and  color percolation on BGW trees} \label{susub:RestrictedDaC-MDM}

Let us now consider the (uniform) DaC percolation model on a BGW tree, conditioned on the case where a mother and child belonging to different percolation clusters cannot have the same color.

The following procedure defines the \emph{restricted DaC model}:
\begin{itemize}
	\item[] {\bf Step 1}:  Perform Bernoulli bond percolation on $\VEC{t}$ (as described in Subsection~\ref{sec:BerPer}). As a result, we obtain a random subgraph of open edges, denoted by $\mathcal{O}$, and let $\mathscr{C}$ denote its set of clusters (referred to as percolation clusters). Note that each cluster $K \in \mathscr{C}$ is, in particular, a subtree rooted at $\emptyset_K$. 
	\item[] {\bf Step 2}: 
	Given an assigned color for the percolation cluster of the initial root, we color the other percolation clusters in the order determined by the lexicographical order of their roots (as defined by the Ulam-Harris notation). The color of $\emptyset_K$ determines the color of every vertex in its cluster $K$. Then, 
	if the color of the mother of $\emptyset_K$ is $i$,
	the color of $K$ is chosen uniformly from the set $[d]\setminus \{i\}$. 
	The resulting coloring $\widetilde{\Pi}$  is defined by a probability measure on $[d]^V$, and the law of  $\widetilde{\Pi}$ is denoted by $\widetilde{\mathbb{P}}_{p, d, \text{unif}}^{\VEC{t}}$.
\end{itemize}
Comparing this construction with one for the DaC model on trees in Subsection~\ref{subsubsec: dac on trees},
the difference lies in the rule for choosing a color for each percolation cluster.

Similarly to the previous subsection, we can compare this percolation model with the mother-dependent mutation model defined in Subsection~\ref{subsub: MDM}. 
Let $ \VEC{u} = (u_1, \ldots, u_d) \in \mathbb{N}_0^d $ and  $|\VEC{u}| = \sum_{j=1}^d u_j$.
In this case, following equation~\eqref{eq: mothdepproba}, conditioned on $\{\xi^{(+, i)}= |\VEC{u}|\}$,  the probability that the root has $\boldsymbol{\xi}^i \coloneqq (\xi^{(i,1)},\xi^{(i,2)},\dots,\xi^{(i,d)})=(u_1,u_2,..,u_d)$ children is given by   
\begin{equation}
	\mathbb{P}\left(\boldsymbol{\xi}^i= \VEC{u} \vert \xi^{(+, i)}= |\VEC{u}| \right)=
	\binom{|\VEC{u}|}{u_i} (1-r)^{u_i}
	r^{|\VEC{u}|-u_i}  
	\binom{|\VEC{u}|-u_i}{u_1, \dots , u_{i-1},u_{i+1}, \dots, u_d}
	\left(\dfrac{1}{d}\right)^{ |\VEC{u}| - u_i }.
\end{equation}
This coincides with the restricted uniform DaC measure $\widetilde{\mathbb{P}}_{p, d, \text{unif}}^{\VEC{t}}$ defined in the step-by-step procedure above, applied to the root and its children at the first generation (see Figure~\ref{fig:11}).

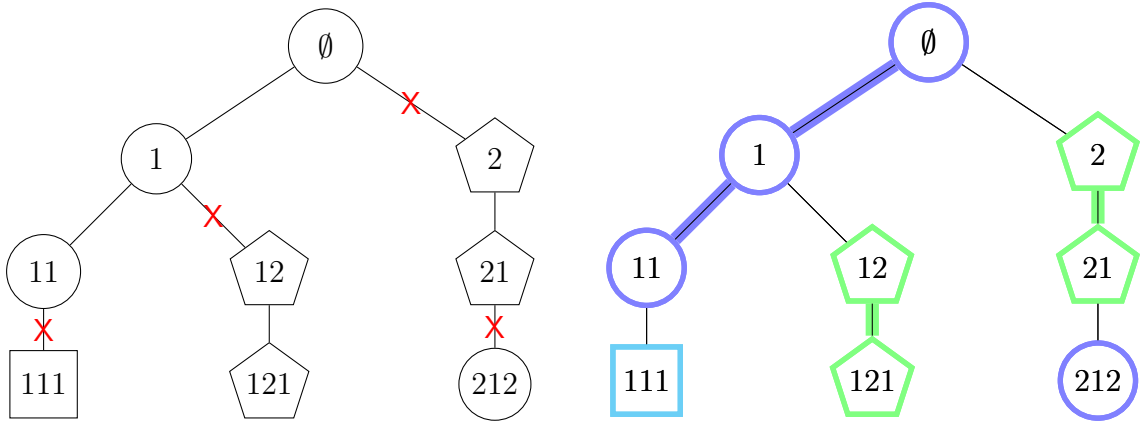
\begin{figure}[h]
	\centering
	\begin{subfigure}{{.45\textwidth}}
		\centering
		\begin{tikzpicture}
			[
			level 1/.style={sibling distance=45mm},
			level 2/.style={sibling distance=30mm},
			level 3/.style={sibling distance=12mm},
			level 4/.style={sibling distance=10mm},]
			\node[circle, draw,inner sep=6](root) { $\emptyset$}
			child {node[circle,draw,inner sep=6] {1}
				child {node[circle,draw,inner sep=5] {11}
					child {node[regular polygon,
						regular polygon sides=4,draw,inner sep=0.2] {111} edge from parent node {\textcolor{red}{{\sf X}}}}
				}
				child {node[regular polygon,
					regular polygon sides=5,draw,inner sep=3] {12} 
					child {node[regular polygon,
						regular polygon sides=5,draw,inner sep=0.2] {121}}edge from parent node {\textcolor{red}{{\sf X}}}} }
			child {node[regular polygon,
				regular polygon sides=5,draw,inner sep=5] {2}
				child {node[regular polygon,
					regular polygon sides=5,draw,inner sep=3] {21}
					child {node[circle,draw,inner sep=3] {212}edge from parent node {\textcolor{red}{{\sf X}}}}}edge from parent node {\textcolor{red}{{\sf X}}}};
		\end{tikzpicture}
	\end{subfigure}
	\begin{subfigure}{{.45\textwidth}}
		\centering
		\begin{tikzpicture}
			[
			level 1/.style={sibling distance=45mm},
			level 2/.style={sibling distance=30mm},
			level 3/.style={sibling distance=12mm},
			level 4/.style={sibling distance=10mm},
			every node/.style={circle, draw=darkgray!50,line width=2pt, minimum size = 1 cm},
			bluen/.style={circle, draw=blue!50,line width=2pt, minimum size = 1 cm},
			greenn/.style={circle, draw=green!50,line width=2pt, minimum size = 1 cm},
			cyann/.style={circle, draw=cyan!50,line width=2pt, minimum size = 1 cm},
			tealn/.style={circle, draw=teal!50,line width=2pt, minimum size = 1 cm},
			oliven/.style={circle, draw=olive!50,line width=2pt, minimum size = 1 cm},
			violetn/.style={circle, draw=violet!50,line width=2pt, minimum size = 1 cm},
			emph/.style={edge from parent/.style={draw,line width=5pt,-,darkgray!50}},
			rede/.style={edge from parent/.style={draw,line width=5pt,-,red!50}},
			bluee/.style={edge from parent/.style={draw,line width=5pt,-,blue!50}},
			greene/.style={edge from parent/.style={draw,line width=5pt,-,green!50}},
			cyane/.style={edge from parent/.style={draw,line width=5pt,-,cyan!50}},
			norm/.style={edge from parent/.style={solid,black,thin,draw}}
			]
			\node[bluen,circle, draw,inner sep=6](root) { $\emptyset$}
			child {node[bluen,circle,draw,inner sep=6] {1}
				child {node[bluen,circle,draw,inner sep=5] {11}
					child {node[cyann,regular polygon,
						regular polygon sides=4,draw,inner sep=0.2] {111} 
					}
				}
				child {node[greenn,regular polygon,
					regular polygon sides=5,draw,inner sep=3] {12} 
					child{node[greenn,regular polygon,
						regular polygon sides=5,draw,inner sep=0.2] {121}}} }
			child {node[greenn,regular polygon,
				regular polygon sides=5,draw,inner sep=5] {2}
				child {node[greenn,regular polygon,
					regular polygon sides=5,draw,inner sep=3] {21}
					child {node[bluen,circle,draw,inner sep=3] {212}}
				}
			};
			\begin{pgfonlayer}{background}
				\node[bluen,circle, draw,inner sep=6](root) { $\emptyset$}
				child[bluee] {node[bluen,circle,draw,inner sep=6] {1}
					child {node[bluen,circle,draw,inner sep=5] {11}
						child[norm] {node[cyann,regular polygon,
							regular polygon sides=4,draw,inner sep=0.2] {111} 
						}
					}
					child[norm] {node[greenn,regular polygon,
						regular polygon sides=5,draw,inner sep=3] {12} 
						child[greene] {node[greenn,regular polygon,
							regular polygon sides=5,draw,inner sep=0.2] {121}}} }
				child {node[greenn,regular polygon,
					regular polygon sides=5,draw,inner sep=5] {2}
					child[greene] {node[greenn,regular polygon,
						regular polygon sides=5,draw,inner sep=3] {21}
						child[norm] {node[bluen,circle,draw,inner sep=3] {212}}
					}
				};
			\end{pgfonlayer}
		\end{tikzpicture}
	\end{subfigure}
	\caption{Comparison between the MDM model and  restricted DaC percolation.} \label{fig:11}
\end{figure}

\subsection{The percolation phase transition from the point of view of neutral mutation models} \label{subsec: phasetran mutation}

For the Bernoulli phase transition on a tree $\mathcal{T}$, the classical interpretation of the percolation supercritical regime corresponds to the existence of an infinite cluster with probability 1. 
From the point of view of population genetics, for a BGW tree with mutations, we  can interpret the root’s cluster as a subtree $\mathcal{T}_{\emptyset}$ consisting of individuals of the same type as the root $\emptyset$. Note that a subtree is connected, so $\mathcal{T}_{\emptyset}$ does not include individuals of the same type as $\emptyset$ who have ancestors of a different type.

In the percolation supercritical regime, the size of $\mathcal{T}_{\emptyset}$ has a positive probability of being of the same order as the size of the full genealogy. If the full genealogical tree is finite, it means that the size of $\mathcal{T}_{\emptyset}$ is proportional to the size of $\mathcal{T}$. On the other hand, if $\mathcal{T}$ is infinite, this means that $\mathcal{T}_{\emptyset}$ is infinite as well.

In the following, we present some percolation results on BGW trees conditioned to be infinite, and we describe them in terms of the mutation models introduced above. 
First, we recall some results of Lyons for percolation in BGW trees and their interpretation in terms of the infinite alleles mutation model and the MDM model. Finally, we will extend Proposition~\ref{prop:Haggstrom} to the DaC model on BGW trees and describe its meaning for the MIM model.

\subsubsection{Percolation phase transition for BGW trees}

We now present a generalization of Proposition~\ref{prop:critical-dtree} and Theorem~\ref{thm:branchingPerc} to BGW trees. 

\begin{theorem}[Lyons, 1990]  \label{thm:lyons}
		Let $\mathcal{T}$ be a BGW tree whose offspring distribution has mean $m > 1$. That is, if $\xi$ is a random variable following the offspring distribution, then $m = \mathbb{E}(\xi) > 1$. Conditioned on non-extinction, we have that $p_c^{\mathcal{T}} = 1/m$ almost surely.
\end{theorem}

In the framework of a BGW process with neutral mutations and infinite alleles, presented in Subsection~\ref{sec:GWmut}, consider the probabilities
\begin{equation} 
\label{eq:probInfiniteTree}
\begin{split}
    q_{\infty} &=  \mathbb{P} 
    \left(  \text{There exists an infinite subtree with the same type} \right), 
    \\
     q_{\emptyset} &=  \mathbb{P} 
	\left(  \mathcal{T}_{\emptyset} \text{ is infinite} \right) ,
\end{split}\end{equation}
where  $\mathcal{T}_{\emptyset}$ corresponds to the subtree of clones of the root $\emptyset$. 

Following the correspondence in Subsection~\ref{subsub:PercolationNeutralMutations}, we interpret Theorem~\ref{thm:lyons} as a phase transition for the probabilities in~\eqref{eq:probInfiniteTree}. 
Specifically,
\begin{equation} \label{eq:phasetransitionq1}
q_{\infty} = 
\begin{cases}
    0 & \text{if } 1 - r \leq p^{\mathcal{T}}_c \\
    1 & \text{if } 1 - r > p^{\mathcal{T}}_c .
\end{cases}
\end{equation}
and 
\begin{equation} \label{eq:phasetransitionq2}
q_{\emptyset}  
\begin{cases}
    = 0 & \text{if } 1 - r \leq p^{\mathcal{T}}_c \\
    > 0 & \text{if } 1 - r > p^{\mathcal{T}}_c .
\end{cases}
\end{equation}

A similar transition holds for the mother-dependent mutation model (MDM) from Subsection~\ref{subsub: MDM}, applying the relation of the MDM model with the restricted DaC model in Subsection~\ref{susub:RestrictedDaC-MDM}.
We can make an interpretation for the probabilities $q^\star_{\infty}$ and $q^\star_{\emptyset}$. We define $q^\star_{\infty}$ in the same way as $q_{\infty}$.  The difference for $q^\star_{\emptyset}$ is that instead of considering $\mathcal{T}_{\emptyset}$, we look at the
subtree
$\mathcal{T}^\star_{\emptyset}$  of individuals of the same type as the root $\emptyset$.
Then, we define
\begin{equation} 
\label{eq:probInfiniteTree-dependentMother}
    q^\star_{\emptyset} =  \mathbb{P} 
    \left( \text{The subtree  } \mathcal{T}^\star_{\emptyset} \text{ is infinite} \right) .
\end{equation}
A phase transition holds as in~\eqref{eq:phasetransitionq1} and~\eqref{eq:phasetransitionq2}.

\subsubsection{Divide and color percolation phase transition for BGW trees}

Finally, we consider DaC percolation on a BGW tree conditioned to be infinite. We will compare this model with the mother-independent mutation model (MIM) introduced in Subsection~\ref{subsub: MIM}. For this purpose, we extend Proposition~\ref{prop:Haggstrom} from deterministic trees to BGW trees.

\begin{theorem}  \label{thm:new}
    Let $\mathcal{T}$ be a BGW tree whose offspring distribution has mean $m > 1$, as in Theorem~\ref{thm:lyons}. Conditionally on non-extinction,
    we have
    \[
    \widetilde{p}_c^{\mathcal{T}} = \dfrac{1-ma_{\emptyset}}{m(1-a_{\emptyset})}
    \qquad \text{ a.s. given non-extinction. }
    \]
\end{theorem}

\begin{proof}
As in the deterministic case, the connected component of a given color containing a given vertex is distributed as in Bernoulli percolation on a BGW tree with percolation parameter $\widetilde{p}=1-(1-p)(1-a_{\emptyset})$. As a consequence, Proposition 5.9 in~\cite[Section 5]{lyons2017probability} applies here as well,  implying that the cluster of the root (by color) will be finite a.s. if $m\widetilde{p}\leq 1$. From this, we deduce the  critical probability for this case. 
\end{proof}

Following the correspondence in Subsection~\ref{subsub:DaC-MIM}, we provide an interpretation of Theorem~\ref{thm:new} for the MIM model with mutation parameter $r \in [0,1]$. We obtain results similar to those in~\eqref{eq:phasetransitionq1} and~\eqref{eq:phasetransitionq2}. In this case, we define the probabilities
\begin{equation} 
	\label{eq:probM-IN}
	\begin{split}
		\widetilde{q}_{\infty} &=     \mathbb{P}^{\mathcal{T}}_{p, 2, 1/d} 
		\left(  \text{There exists an infinite subtree of the same type as the root} \right) ,
		\\
		\widetilde{q}_{\emptyset} &=     \mathbb{P}_{p, 2, 1/d}^\mathcal{T} 
		\left(  \widetilde{\mathcal{T}}_{\emptyset} \text{ is infinite} \right) ,
\end{split}\end{equation}
where $\widetilde{\mathcal{T}}_{\emptyset}$ corresponds to the
subtree of individuals who are of the same type as the root and have no ancestors of other types.

We then interpret Theorem~\ref{thm:new} as a phase transition for the probabilities in~\eqref{eq:probM-IN}. 
Specifically,
\begin{equation} \label{eq:phasetransitionq3}
	\widetilde{q}_{\infty} = 
	\begin{cases}
		0 & \text{if } 1 - r \leq \widetilde{p}^{\mathcal{T}}_c \\
		1 & \text{if } 1 - r > \widetilde{p}^{\mathcal{T}}_c ,
	\end{cases}
\end{equation}
and 
\begin{equation} \label{eq:phasetransitionq4}
	\widetilde{q}_{\emptyset}  
	\begin{cases}
		= 0 & \text{if } 1 - r \leq \widetilde{p}^{\mathcal{T}}_c \\
		> 0 & \text{if } 1 - r > \widetilde{p}^{\mathcal{T}}_c .
	\end{cases}
\end{equation}

\section*{Acknowledgements}
MCF's research is supported by the UNAM-PAPIIT grant IN109924. SHT would like to acknowledge the support of the UNAM-PAPIIT grant IA103724.

The authors would like to thank the research network \href{https://sites.google.com/view/aleatorias-normales/home}{Aleatorias \& Normales}, which connects female Latin American researchers in probability and statistics. The discussions facilitated by the \emph{Aleatorias \& Normales Seminar} served as the starting point for this project.

Finally, the authors would like to thank an anonymous referee for their detailed remarks and questions.

\bibliography{biblio}
\bibliographystyle{plain}

\end{document}